\documentclass[a4paper,abstracton]{scrartcl}
\usepackage{amsfonts}
\usepackage{amssymb}
\usepackage{amsmath}
\usepackage{amsthm}

\usepackage{mathtools}
\usepackage{multirow}

\usepackage[ruled,vlined,linesnumbered]{algorithm2e}

\usepackage{graphicx}
\usepackage{color}

\usepackage{subfigure}

\usepackage{a4wide}

\newtheorem{theorem}{Theorem}

\newtheorem{corollary}[theorem]{Corollary}

\usepackage{authblk}

\allowdisplaybreaks

\newcommand{\X}{{\mathcal{X}}}
\newcommand{\cU}{{\mathcal{U}}}

\newcommand{\cS}{{\mathcal{S}}}	
	
\newcommand{\cK}{{\mathcal{K}}}

\title{Robust Combinatorial Optimization Problems Under Budgeted Interdiction Uncertainty}
	
\author{Marc Goerigk}
\author{Mohammad Khosravi\footnote{Corresponding author. Email: mohammad.khosravi@uni-passau.de}}
	
\affil{Business Decisions and Data Science, University of Passau,\authorcr Dr.-Hans-Kapfinger-Stra{\ss}e 30, 94032 Passau, Germany}
	
\date{}

\begin{document}

\maketitle

\begin{abstract}
In robust combinatorial optimization, we would like to find a solution that performs well under all realizations of an uncertainty set of possible parameter values. How we model this uncertainty set has a decisive influence on the complexity of the corresponding robust problem. For this reason, budgeted uncertainty sets are often studied, as they enable us to decompose the robust problem into easier subproblems. We propose a variant of discrete budgeted uncertainty for cardinality-based constraints or objectives, where a weight vector is applied to the budget constraint. We show that while the adversarial problem can be solved in linear time, the robust problem becomes NP-hard and not approximable. We discuss different possibilities to model the robust problem and show experimentally that despite the hardness result, some models scale relatively well in the problem size.
\end{abstract}


\noindent\textbf{Acknowledgements:} Supported by the Deutsche Forschungsgemeinschaft (DFG) through grant 459533632.

\section{Introduction}\label{sec:introduction}

Uncertainty can manifest in various forms, such as imprecise data or the inherent unpredictability of the future. A notable case study utilizing linear programs \cite{ben2000robust} demonstrated that even slight changes in problem data can significantly shift an optimal solution towards infeasibility, rendering it practically useless. Consequently, a range of decision-making approaches under uncertainty have been developed, including stochastic programming \cite{kail2005stochastic}, fuzzy optimization \cite{lodwick2010fuzzy}, and robust optimization \cite{ben2009robust}. Often, such approaches make the resulting decision-making problems more challenging to solve than their nominal counterparts.
The focus of this paper is on robust combinatorial decision problems, which have the distinct advantage that a probability distribution on the uncertain data does not need to be known. More formally, consider some nominal combinatorial problem
\begin{align*}
& \min \sum_{i\in[n]} c_i x_i \tag*{(Nom)}\\
& \text{s.t.} \; \pmb{x} \in \X \subseteq \{0,1\}^n
\end{align*}
where we write vectors in bold and use the notation $[n]$ to denote sets $\{1,\ldots,n\}$. In addition, assume that the data $\pmb{c}$ in the objective function is not known exactly. Given a set of possible data values $\cU$, the classic min-max approach to robust optimization is to solve the problem
\begin{equation}
\min_{\pmb{x}\in\X} \max_{\pmb{c} \in \cU} \sum_{i\in[n]} c_ix_i \tag*{(RO)}
\end{equation}
Many more variants of robust optimization problems exist, see e.g. \cite{goerigk2016algorithm,kasperski2016robust,buchheim2018robust} for an overview. What they have in common is that a set $\cU$ containing all scenarios can be formulated by the decision maker, and is made available to the optimization problem. Data-driven robust optimization \cite{bertsimas2018data} aims at automating this step by formulating suitable uncertainty sets based on available data (e.g., by using on the risk preference of the decision maker).

There is typically a trade-off between the modeling capabilities of the uncertainty set $\mathcal{U}$ and the complexity of the resulting problem. A discrete scenario set $\mathcal{U}$ offers broad flexibility as it allows direct utilization of any amount of historical data observations in the model. However, it comes with a drawback that the robust versions of relevant combinatorial problems are already computationally difficult (NP-hard) even when considering only two scenarios \cite{kasperski2016robust}. Representing $\mathcal{U}$ using a general polyhedron, defined by its inner or outer description, suffers from the same limitation \cite{goerigk2022robust}.

A significant breakthrough was made with the introduction of budgeted uncertainty sets, also known as the Bertsimas-Sim approach \cite{bertsimas2003robust,bertsimas2004price}. This approach addresses an uncertain linear objective $\pmb{c}^\intercal \pmb{x}$, where each coefficient $i \in [n]$ is bounded by a lower bound $\underline{c}_i$ and an upper bound $\overline{c}_i$. Moreover, only a fixed integer $\Gamma$ of coefficients are allowed to deviate simultaneously from their lower to upper bounds. In other words, $\mathcal{U}$ incorporates a cardinality constraint of the following form:
\[\cU_\Gamma =  \Bigg\{ \pmb{c} \in \mathbb{R}^n \colon \exists \pmb{\delta} \in \{0,1\}^n \; \text{s.t.} \; c_i = \underline{c}_i + (\overline{c}_i - \underline{c}_i)\delta_i , \sum_{i\in[n]} \delta_i \leq \Gamma \Bigg\} \]
The introduction of this simple idea has had a profound impact on the field of robust optimization. The two papers that presented this idea continue to be widely cited, highlighting their significance. The appeal of uncertainty sets of this nature lies in their simplicity and intuitive nature. Furthermore, it has been demonstrated that the robust min-max problem can be decomposed into a manageable number (specifically, $O(n)$) of nominal-type problems (Nom). This decomposition allows for increased modeling flexibility without incurring significant computational complexity. If the nominal problem can be solved in polynomial time, the corresponding robust problems can be solved in polynomial time as well.

The advantages offered by budgeted uncertainty sets have resulted in their widespread and varied applications to real-world problems. These applications encompass a range of domains, including portfolio management \cite{bertsimas2008robust}, wine grape harvesting \cite{bohle2010robust}, supply chain control \cite{bertsimas2006robust}, furniture production planning \cite{alem2012production}, train load planning \cite{bruns2014robust}, and many others. The versatility of budgeted uncertainty sets has made them a valuable tool in addressing uncertainty and optimizing decision-making in numerous practical scenarios. 

A noteworthy characteristic of $\mathcal{U}_\Gamma$ is that if $\Gamma$ is an integer, we can utilize continuous deviations $\pmb{\delta} \in [0,1]^n$ without altering the problem. This is due to the fact that when finding an optimal strategy for the adversary in the problem $\max_{\pmb{c}\in\mathcal{U}} \pmb{c}^\intercal \pmb{x} $ given a fixed solution $\pmb{x}$, it is sufficient to sort the items chosen by $\pmb{x}$ based on the potential cost deviation $\overline{c}_i - \underline{c}_i$, and select the $\Gamma$ largest values. Consequently, the equivalence between ``discrete'' and ``continuous'' budgeted uncertainty holds. However, this equivalence does not generally hold in the case of multi-stage robust problems, where recourse actions can be taken after the cost scenario has been revealed (see, e.g., the discussion in \cite{goerigk2022complexity}).

The effectiveness of budgeted uncertainty sets has led to the emergence of various variants and generalizations of this approach. In the paper by \cite{bertsimas2004robust}, norm-based uncertainty sets were introduced. It was demonstrated that the traditional budgeted uncertainty set can be constructed using a specific norm known as the D-norm. Another variant is multi-band uncertainty \cite{busing2012new}, which involves a system of deviation values $d^1_{ij} < d^2_{ij} < \ldots < d^K_{ij}$ with both lower and upper bounds on the number of possible deviations from each band $k \in [K]$. In variable budgeted uncertainty \cite{poss2013robust}, the number of deviations $\gamma$ taken into account may depend on the size $\lVert\pmb{x}\rVert_1$ of the solution $\pmb{x}$ for which the adversarial problem is being solved.

Additionally, there is knapsack uncertainty \cite{poss2018robust}, which can be represented as follows:
\[ \cU_{knap} = \Bigg\{ \pmb{c} \in \mathbb{R}^n : \exists \pmb{\delta} \in  [0,1]^n \; \text{s.t.} \; c_i = \underline{c}_i + (\overline{c}_i - \underline{c}_i) \delta_i , \sum_{i\in[n]} a_{ji} \delta_i \leq b_j , j \in [m] \Bigg\} \]
Here, the set of possible scenarios is bounded by $m$ linear knapsack constraints. When the value of $m$ is fixed, similar results to those obtained for the original set $\mathcal{U}_\Gamma$ can be derived. A special case of this type of set is locally budgeted uncertainty, see \cite{goerigk2021robust,yaman2023short}, where each of the knapsack constraints affects a subset of variables, and these subsets are disjoint between constraints. These variants and generalizations of budgeted uncertainty sets provide additional flexibility and adaptability to various problem settings, enhancing the robustness of decision-making under uncertainty.

In this paper we consider a new type of uncertainty set, applicable to an objective or constraints that involve the cardinality\footnote{As $\pmb{x}$ is binary, it corresponds to a subset $X$ of $[n]$, where $X=\{i\in[n]: x_i = 1\}$. The notion of cardinality refers to $|X|=\|\pmb{x}\|_1$.} $\lVert\pmb{x}\rVert_1$, e.g., to problems where the task is to maximize the size of a set, or where this cardinality is not allowed to fall below a certain threshold. The motivation to consider such sets comes from a real-world problem involving the composition of teams to take on a set of jobs under uncertain skill requirements (see \cite{anoshkina2019technician,anoshkina2020robust}). In such problems, one would like to compose teams that can take on the maximum possible number of jobs. From an adversarial perspective, the task is to change the job skill requirements in a way that minimizes the number of jobs that can be carried out successfully. From a more theoretical perspective, the study of robust combinatorial problems often makes use of selection-type problems (see, e.g., \cite{averbakh2001complexity,dolgui2012min,deineko2013complexity,kasperski2015approximability}).
In the most basic form, the selection problem requires us to select $p$ out of $n$ possible items, i.e., to solve
\[\min \Bigg\{\pmb{d}^\intercal \pmb{x} \colon \sum_{i\in[n]} x_i \geq p \; , \; \pmb{x} \in \{0,1\}^n\Bigg\}\]
with known costs $\pmb{d} \in \mathbb{R}^n_+$. While this nominal problem is trivial to solve, treating robust variants becomes more complex. A new perspective on problems of this type is to locate the uncertainty not (only) on the item costs; instead, items have different degrees of reliability, and an adversary tries to violate the constraint $\sum_{i\in[n]} x_i \geq p$.


Motivated by these two problems, but being applicable to a wider range of problems as well, the ``budgeted interdiction'' approach that we thus propose is to consider uncertainty sets of the form
\[\cU = \Bigg\{\pmb{c} \in \{0,1\}^n \colon \sum_{i\in[n]} w_i c_i \leq B \Bigg\}\]
with $\pmb{w}\in\mathbb{N}^n$ and $B \in\mathbb{N}$. The adversary can therefore interdict a solution (i.e., let items fail), but has a specified budget for this purpose. Throughout the paper, we assume that each $w_i$ is not larger than $B$; otherwise, its coefficient cannot be attacked and is therefore not uncertain.
The uncertainty can affect a cardinality objective function $(\pmb{1}-\pmb{c})^\intercal \pmb{x}$ that should be maximized, or a cardinality constraint $(\pmb{1}-\pmb{c})^\intercal \pmb{x} \ge p$. Note that in the corresponding nominal problems, vector $\pmb{c}$ is not present, but is introduced in the robust problem to model the uncertainty of the vector $\pmb{1}$.
Cardinality constraints also play a role in many optimization problems that allow a cut-based formulation. For example, the shortest path problem can be written as
\begin{align*}
\min\ & \sum_{i\in E} d_e x_e \\
\text{s.t.} & \sum_{e\in \delta(S)} x_e \ge 1 & \forall S \in \mathcal{S}\\
& x_e \in \{0,1\}
\end{align*}
with $\mathcal{S} = \{S\subseteq V \colon s \in S, t \notin S\}$. Cut-based problem formulations are also used for the generalized Steiner tree, spanning tree, feedback vertex set, or traveling salesperson problems \cite{kortevygen}.

Observe that this definition of uncertainty set is essentially the budgeted uncertainty set $\cU_\Gamma$ ``upside down'': while $\cU_\Gamma$ has a bound on the number of coefficients that can deviate and the effect of deviation is given by some parameter $\overline{c}_i-\underline{c}_i$, here we want to maximize the number of deviations and each deviation has a cost parameter $w_i$.
Note that different to $\cU_{knap}$, there is a single budget constraint, we consider a discrete instead of continuous deviation, and in particular, the vector $\pmb{c}$ is binary.

As an example, consider the selection problem
\[\min \Bigg\{2x_1+3x_2+4x_3+5x_4 \colon \sum_{i\in[4]} x_i \geq 1 , \pmb{x} \in \{0,1\}^4 \Bigg\}\]
where the cardinality constraint $\sum_{i\in[4]} x_i \geq 1$ is uncertain and thus can be attacked by an adversary. The cardinality constraint of the robust counterpart of this example becomes
\[\min_{\pmb{c}\in\cU} (\pmb{1}-\pmb{c})^\intercal\pmb{x} = \sum_{i\in[n]} x_i - \max_{\pmb{c}\in\cU} \pmb{c}^\intercal\pmb{x} =  \sum_{i\in[n]} x_i - \phi(\pmb{x}) \geq 1\]
where the function $\phi(\pmb{x})$ represents the number of items that can fail. To further illustrate this setting, let us assume that
\[\phi(\pmb{x}) = \max \Bigg\{\sum_{i\in[n]} x_i c_i \colon 3c_1+7c_2+4c_3+10c_4 \leq 10, \pmb{c}\in\{0,1\}^4 \Bigg\}\]
that is, in the definition of $\cU$, we use $B = 10$ and $w = (3,7,4,10)^\intercal$. A possible solution to the robust problem is to pick items 1, 2 and 3 at cost $2 + 3 + 4 = 9$. The adversary can attack items 1 and 3, but does not have sufficient budget to let all three items fail. An even better solution is to pick items 1 and 4 at cost $2 + 5 = 7$. In this case, the adversary can only attack one of the two items.

The remainder of this paper is structured as follows. In Section~\ref{sec:problems}, we discuss the complexity of the robust problem with budgeted interdiction uncertainty, and prove that the problem is not approximable. Furthermore, we provide five compact formulations to solve problems with cardinality constraints under interdiction uncertainty set in Section~\ref{sec:formulations}. Experimental results illustrating the performance of the models for the selection, job assignment and 2-edge-connected subgraph problems are collected in Section~\ref{sec:experiments}. We summarize our findings and pointing out further research questions in Section~\ref{sec:conclusions}. The detailed information on how to model the compact formulation of both job assignment and cut-based problems are provided in Appendix~\ref{app:jobassignment} and \ref{app:cutbased}, respectively.

\section{Complexity Analysis}\label{sec:problems}

In order to check the complexity level of the robust selection problem under interdiction budgeted uncertainty, we first need to introduce the compact formulation of it, thus we have
\begin{align*}
\min\ &\sum_{i\in[n]} d_i x_i \tag*{(ROSel)}\\
\text{s.t. } & \sum_{i\in[n]} x_i - \phi(\pmb{x})\ge p & \forall \pmb{c}\in\cU \\
& \pmb{x}\in\{0,1\}^n
\end{align*}
where the adversary problem is
\[\phi(\pmb{x}) = \max \Bigg\{\sum_{i\in[n]} x_i c_i \colon \sum_{i\in[n]} w_i c_i \leq B \; , \; c_i \in \{0,1\} \; \forall i \in [n] \Bigg\} \]
for a given $\pmb{x} \in \{0,1\}^n$. There is a trivial algorithm to solve this problem; namely, we sort items $i$ with $x_i = 1$ by non-decreasing weight $w_i$, and pack items in this order until the budget $B$ cannot accommodate any further items. Hence, the adversarial problem can be solved in $O(n)$ time (as it is not necessary to sort the complete vector, see, e.g. \cite[Chapter 17.1]{kortevygen}). Now we show that the decision version of the robust selection problem with interdiction uncertainty affecting the constraints (ROSel) is hard.

\begin{theorem}\label{th1}
The following decision problem is NP-complete: Given $\pmb{d}\in\mathbb{N}^n$, $\pmb{w}\in\mathbb{N}^n$, $B\in\mathbb{N}$, and $V\in\mathbb{N}$, is there a vector $\pmb{x}\in\{0,1\}^n$ with $\sum_{i\in[n]} x_i - \phi(\pmb{x}) \ge 1$ and $\sum_{i\in[n]} d_i x_i \le V$?
\end{theorem}
\begin{proof}
Observe that it is trivial to check if $\sum_{i\in[n]} x_i - \phi(\pmb{x}) \ge 1$ and $\sum_{i\in[n]} d_i x_i \le V$ for a given $\pmb{x}$, which means that the decision problem is indeed in NP.

To show NP-completeness, we make use of the partition problem: Given positive integers $v_1,\ldots,v_n$, is there a set $S \subseteq [n]$ such that $\sum_{i\in S} v_i= V$ with $V=\sum_{i\in[n]} v_i/2$?

Given such an instance of the partition problem, we construct a robust problem with budgeted interdiction in the following way. Set $d_i = w_i = v_i$ and $B = V - 1$. Then the constraint
\[\sum_{i\in[n]} x_i - \phi(\pmb{x}) \geq 1 \]
with
\[\phi(\pmb{x}) = \max \Bigg\{ \sum_{i\in[n]} x_i c_i \colon \sum_{i\in[n]} v_i c_i \leq V - 1 \Bigg\} \]
requires us to pack items of total weight strictly greater than $V - 1$ to avoid having all items interdicted. This means that the partition problem is a Yes-instance if and only if there is a feasible solution $\pmb{x}\in\{0,1\}^n$ with objective value less or equal to $V$. As the partition problem is well-known to be NP-complete \cite{garey1979computers}, the claim follows.
\end{proof}

This brief analysis shows that we lose a main advantage of classic budgeted uncertainty, where the robust problem can be decomposed into a set of nominal problems. Note that Theorem~\ref{th1} applies to optimization problems with an uncertain cardinality constraint and an objective $\sum_{i\in[n]} d_i x_i$ that should be minimized, but it also applies to the case of having one linear constraint $\sum_{i\in[n]} d_i x_i \le V$ and an uncertain cardinality objective that should be maximized. In particular, in the latter case this means that it is NP-complete to find a solution with a non-zero objective value; in other words, it is not possible to find a polynomial-time approximation algorithm for this setting, unless P=NP. Hence we conclude the following result.

\begin{corollary}
The optimization problem $\max_{\pmb{x} \in \X} \min_{\pmb{c} \in \cU} \sum_{i\in[n]} (1-c_i)x_i$ is not approximable, even if $\max_{\pmb{x} \in \X} \sum_{i\in[n]} x_i$ can be solved in polynomial time.
\end{corollary}
\begin{proof}
Given a partition problem as in the proof of Theorem~\ref{th1}, set $\X=\{ \pmb{x}\in\{0,1\}^n : \sum_{i\in[n]} v_i x_i \le V\}$. Then there is a solution with objective value greater or equal to one if and only if the partition problem is a Yes-instance. Hence, there cannot be an $\alpha$-approximation for any $\alpha>0$, unless P=NP.
\end{proof}

%

\section{Model Formulations}
\label{sec:formulations}

In this section, we introduce five compact formulations of the robust problem, where we focus on an uncertain cardinality constraint $\sum_{i\in[n]} x_i \ge p$ for ease of presentation.
Additional constraints on $\pmb{x}$ may be considered, which are assumed to be modeled indirectly in the set $\X\subseteq\{0,1\}^n$. That is, we consider reformulations of the following type of robust problem with cardinality constraints:
\begin{align*}
\min\ &\sum_{i\in[n]} d_i x_i \\
\text{s.t. } & \sum_{i\in[n]} (1-c_i) x_i \ge p & \forall \pmb{c}\in\cU \\
& \pmb{x}\in\X
\end{align*}
where the nominal problem corresponds to the case $\pmb{c}=\pmb{0}$. In addition, without loss of generality, we assume that the items are sorted based on their weights ($w_i$), non-decreasingly.

\subsection{IP-1}

The first idea to find a compact formulation of the problem is only applicable to the case $p=1$ with integer weights $\pmb{w}$. This means we only need to have one item after the adversary attacks, a case that remains hard, as Theorem~\ref{th1} shows. Therefore, it suffices to pack items with minimum cost whose total weight strictly exceeds the adversarial budget $B$. This idea can be formulated as follows:
\begin{align*}
\min & \sum_{i\in[n]} d_i x_i \tag{IP-1}\\
\text{s.t.} & \sum_{i\in[n]} w_i x_i \geq B + 1\\
&\pmb{x}\in\X
\end{align*}

\subsection{IP-2}

We now consider the general case of arbitrary values for $p$. As noted, the adversarial problem $\phi(\pmb{x})$ can be solved in polynomial time by packing items with smallest weight first. Therefore, we introduce variables $\lambda_k\in\{0,1\}$ for all $k\in[n]$, where $\lambda_k$ is active if and only if we attack the first $k$ items (note that the case $k=0$ can be ignored, as we can always attack at least one item, due to each $w_i$ being not larger than $B$). An attack only incurs costs on the interdiction budget if $x_i=1$. Hence, we obtain the following integer program:
\begin{align}
\phi(\pmb{x}) = \max & \sum_{k\in[n]} (\sum_{i\in[k]} x_i) \lambda_k \label{phi-2-1}\\
\text{s.t.} & \left(\sum_{i\in[k]} w_i x_i - B\right) \lambda_k \leq 0 & \forall k \in [n] \label{phi-2-2}\\
& \sum_{k\in[n]} \lambda_k \leq 1 \label{phi-2-3}\\
& \lambda_k \in\{0,1\} & \forall k \in [n] \label{phi-2-4}
\end{align}
By Constraint~\eqref{phi-2-3}, we can only choose one of the candidate attacks represented by $\lambda_k$. Due to Constraint~\eqref{phi-2-2}, we cannot use attack $\lambda_k$ if $\sum_{i\in[k]} w_ix_i > B$. It is easy to see that we can relax the integrality constraints of $\lambda_k$, which gives an LP formulation for $\phi(\pmb{x})$.
By using linear programming duality, we thus can obtain the formulation for the robust problem under budgeted interdiction uncertainty:
\begin{align*}
\min & \sum_{i\in[n]} d_i x_i \\
\text{s.t.} & \sum_{i\in[n]} x_i - \beta \geq p\\
& \sum_{i\in[k]} (w_i \alpha_k - 1) x_i \geq B \alpha_k - \beta & \forall k \in [n] \\
& \alpha_k \geq 0 & \forall k \in [n]\\
& \beta \geq 0 \\
& \pmb{x}\in\X
\end{align*}
This formulation is nonlinear due to the products $(w_i \alpha_k - 1) x_i$. As $\pmb{x}$ is binary, we can linearize the model using $\mu_{ik} = \alpha_k x_i$ for all $k\in[n]$, $i\in[k]$ as follows:
\begin{align*}
\min & \sum_{i\in[n]} d_i x_i  \tag{IP-2}\\ 
\text{s.t.} & \sum_{i\in[n]} x_i - \beta \geq p\\
& \sum_{i\in[k]} (w_i \mu_{ik} - x_i) \geq B \alpha_k - \beta & \forall k \in [n] \\
& \mu_{ik} \leq k x_i & \forall k\in[n], i\in[k]\\
& \mu_{ik} \leq \alpha_k & \forall k\in[n], i\in[k]\\
& \mu_{ik} \geq 0 & \forall k\in[n], i\in[k]\\
& \alpha_k \geq 0 & \forall k \in [n]\\
& \beta \geq 0 \\
& \pmb{x}\in\X
\end{align*}

\subsection{IP-3}

We now consider a third option to model the robust problem. As an alternative formulation for IP-2, here the adversarial problem is obtained by considering the ratio of items weight divided by the budget $B$. Thus the model is as follows:
\begin{align*}
\phi(\pmb{x}) = \max & \sum_{k\in[n]} \Bigg( 1 - \Bigg\lfloor \frac{\sum_{i\in[k]} w_i x_i}{B+\epsilon} \Bigg\rfloor \Bigg) \Bigg( \sum_{i\in[k]} x_i \Bigg) \lambda_k \\
\text{s.t.} & \sum_{k\in[n]} \lambda_k \leq 1 \\
& \lambda_k \geq 0 & \forall k \in [n]
\end{align*}
where $0<\epsilon<1$. Observe that $1 - \lfloor \frac{\sum_{i\in[k]} w_i x_i}{B+\epsilon}\rfloor \le 0$ if and only if $\sum_{i\in[k]}w_ix_i > B$.
Analogously to IP-2, $\phi(\pmb{x})$ represents the highest number of items that can be attacked by the adversary. By using linear programming duality, we can find the following formulation for the robust problem:
\begin{align*}
\min & \sum_{i\in[n]} d_i x_i \\
\text{s.t.} & \sum_{i\in[n]} x_i - \alpha \geq p\\
& \alpha - \Bigg( 1 - \Bigg\lfloor \frac{\sum_{i\in[k]} w_i x_i}{B+\epsilon} \Bigg\rfloor \Bigg) \Bigg( \sum_{i\in[k]} x_i \Bigg) \geq 0 & \forall k \in [n]\\
& \alpha \geq 0 \\
&\pmb{x}\in\X
\end{align*}
The formulation is not linear. To eliminate the floor function, we introduce a new integer variable $y_k$ for $k\in[n]$ with
\begin{align*}
&\Bigg( \frac{\sum_{i\in[k]} w_i x_i}{B+\epsilon} \Bigg) \ge y_k  &\forall k \in [n]
\end{align*}
This leads to products $x_i y_k$ which are linearized using additional variables $z_{ik}$ with
\begin{align*}
&\sum_{i\in[k]} x_i y_k = z_{ik} & \forall k \in [n]
\end{align*}
A compact formulation of the robust problem under knapsack uncertainty is hence as follows:
\begin{align*}
\min & \sum_{i\in[n]} d_i x_i \tag{IP-3}\\
\text{s.t.} & \sum_{i\in[n]} x_i - \alpha \geq p\\
& \sum_{i\in[k]} (x_i - z_{ik}) \leq \alpha & \forall k \in [n] \\
& z_{ik} \leq y_k & \forall k\in[n], i\in[k] \\
& z_{ik} \leq \Bigg\lfloor \frac{\sum_{j\in[n]} w_j}{B+\epsilon} \Bigg\rfloor x_i & \forall k\in[n], i\in[k] \\
& (B+\epsilon)y_k \leq \sum_{i\in[k]} w_i x_i & \forall k\in[n] \\
& z_{ik} \geq 0 & \forall k\in[n], i\in[k] \\
& \alpha \geq 0 \\
& \pmb{y} \in \mathbb{Z}^n \\
&\pmb{x}\in\X
\end{align*}

\subsection{IP-4}

In the following, we assume that $\pmb{w}$ is integer.
Let $W_k(\pmb{x})$ be the smallest required weight to attack at least $k$ items of $\pmb{x}$. If this is not possible then $\sum_{i\in[n]} x_i < k$ and we set $W_k(\pmb{x})=\infty$. To calculate this value, consider the following selection problem:
\begin{align*}
W_k(\pmb{x}) = \min\ & \sum_{i\in[n]} w_i z_i \\
\text{s.t. } & \sum_{i\in[n]} z_i \ge k \\
& z_i \le x_i & \forall i\in[n] \\
& \pmb{z} \in\{0,1\}^n
\end{align*}
where we define the minimum over an empty set to be infinity. This allows us to reformulate $\phi(\pmb{x})$ as a minimization problem in the following way:
\[ \phi(\pmb{x}) = \min\left\{\sum_{k\in[n]} y_k : (B+1)y_k + W_k(\pmb{x}) \ge B + 1\ \forall k\in[n], \pmb{y}\in\{0,1\}^n\right\} \]
where $\sum_{k\in[n]} y_i$ represent the maximum number of items that can be attacked by the adversary. Therefore, if $W_k(\pmb{x}) \le B$, we need to set $y_k=1$ to have the constraint fulfilled; otherwise, we can choose $y_k=0$. As $\phi(\pmb{x})$ is expressed as a minimization problem, the robust problem becomes
\begin{align*}
\min\ &\sum_{i\in[n]} d_i x_i \\
\text{s.t. } & \sum_{i\in[n]} x_i \ge p + \sum_{k\in[n]} y_k \\
& (B+1)y_k + W_k(\pmb{x}) \ge B + 1 & \forall k\in[n] \\
& \pmb{y}\in\{0,1\}^n \\
& \pmb{x}\in\X
\end{align*}
To replace $W_k(\pmb{x})$, note that we can relax variables $z_i$ without changing the value of the corresponding minimization problem. This means that we can use linear programming duality again to arrive at the following problem formulation:
\begin{align}
\min & \sum_{i\in[n]} d_i x_i \label{ip41}\\
\text{s.t.} & \sum_{i\in[n]} x_i \geq p + \sum_{k\in[n]} y_k \label{ip42}\\
& (B + 1) y_k + k\alpha^k - \sum_{i\in[n]} \beta^k_ix_i \geq B +1 & \forall k \in [n] \label{ip43}\\
& \alpha^k - x_i \beta^k_i \leq w_i & \forall k\in[n], i\in[n] \label{ip44}\\
& \alpha^k \geq 0 & \forall k\in[n] \label{ip45}\\
& \beta^k_i \geq 0 & \forall k\in[n], i\in[n] \label{ip46} \\
& \pmb{y}\in\{0,1\}^n \\
& \pmb{x}\in\X
\end{align}
To avoid the non-linearity between $\beta^k_i$ and $x_i$, we increase the weight $w_i$ sufficiently far if $x_i=0$ so that it will not be part of an attack within the available budget $B$. That is, we replace Constraint~\eqref{ip44} with
\[ \alpha^k - \beta^k_i \leq w_i + (B+1)(1-x_i) \ \forall k\in[n], i\in[n] \]
and Constraint~\eqref{ip43} with
\[ (B + 1) y_k + k\alpha^k - \sum_{i\in[n]} \beta^k_i \geq B +1 \ \forall k \in [n] \]
The resulting model is called (IP-4).

\subsection{IP-5}

In this final model, we again make the assumption that weights $\pmb{w}$ are integer. For fixed $k\in[n]$, consider the constraint
\[(B+1)(1-x_k+y_k) \ge B + 1 - \sum_{i\in[k]} w_i x_i \]
with binary variable $y_k$, which shows the attack of the adversary.
As the adversary will attack items in the order from $1$ to $n$ if possible, the term $\sum_{i\in[k]} w_ix_i$ gives the required budget to attack all items that $\pmb{x}$ packed up to and including item $k$. This means that if $x_k=1$ and $\sum_{i\in[k]} w_ix_i \le B$, then item $k$ will be attacked. In the constraint, this corresponds to the case that $(B+1)y_k \ge B+1 - W$ with $W\le B$, so $y_k=1$ is the only feasible choice. On the other hand, if $x_k=0$ or $\sum_{i\in[k]} w_ix_i \ge B+1$, it is possible to set $y_k=0$.

This discussion shows that the following compact formulation for the robust problem with budgeted interdiction uncertainty is correct:
\begin{align*}
\min & \sum_{i\in[n]} d_i x_i \tag{IP-5}\\
\text{s.t.} & \sum_{i\in[n]} x_i \geq p + \sum_{k\in[n]} y_k \\
& (B+1)(1-x_k+y_k) \ge B + 1 - \sum_{i\in[k]} w_i x_i & \forall k \in [n] \\
& \pmb{y}\in\{0,1\}^n \\
& \pmb{x}\in\X
\end{align*}

\subsection{Model Comparison}

In Table~\ref{tab:comparison}, we summarize the five proposed models.

\begin{table}[htb]
	\begin{center}
		\begin{tabular}{|r|r|r|r|r|}
			\hline
			Model & \; \# Cont. V. & \; \# Disc. V. & \; \# Con. & \; Requirements \\
			\hline 
			IP-1 & $0$ & $0$ & 0 & $p=1$, $\pmb{w}\in\mathbb{Z}^n_+$ \\[1.2ex]
			IP-2 & $\frac{1}{2}n^2 + \frac{3}{2}n+1$ & 0 & $n^2+2n$ &  \\[1.2ex]
			IP-3 & $\frac{1}{2}n^2 + \frac{1}{2}n+1$ & $n$ & $n^2+3n$ & \\[1.2ex]
			IP-4 & $n^2+n$ & $n$ & $n^2+n$ & $\pmb{w}\in\mathbb{Z}^n_+$ \\[1.2ex]
			IP-5 & 0 & $n$ & $n$ & $\pmb{w}\in\mathbb{Z}^n_+$ \\
			\hline
		\end{tabular}
		\caption{Comparison of model sizes.}\label{tab:comparison}
	\end{center}
\end{table}

Column ``\# Cont. V.'' gives the additional number of continuous variables, while ``\# Disc. V.'' gives the additional number of discrete variables in the model (beyond variables $\pmb{x}\in\X$. Column ``\# Con.'' shows the additional number of constraints that are created in comparison to the nominal model.

Some models require $\pmb{w}$ to be integer. This property can be replaced with the more general requirement that we need to be able to determine some value $\epsilon>0$ in polynomial time such that for each $\pmb{c}\in\{0,1\}^n$, either $\pmb{w}^t\pmb{c}\le B$ or $\pmb{w}^t\pmb{c} \ge B+\epsilon$ holds. For integer weights, $\epsilon=1$ clearly fulfills this property.

Having five candidate models available (four of which are general), a natural question is to compare the strengths of their respective linear programming relaxations: Is it possible that one model dominates the other in the sense that any feasible solution of the linear programming relaxation of the first model corresponds to a feasible solution of the linear programming relaxation of the second model with the same objective value? As we show experimentally in the next section, this is not the case. For any pairwise comparison of models, there are instances where the LP relaxation of one model outperforms the LP relaxation of the other.

\section{Experiments}
\label{sec:experiments}

In this section we check the performance of our IPs to solve three different problems under budgeted interdiction uncertainty. We start with the selection problem as the models are obtained with respect to this problem. We also show the results of our approaches applied to the job assignment and 2-edge-connected subgraph problem.

To evaluate solution times for each problem, and in addition to presenting the normal solution times, we use performance profiles as introduced in \cite{dolan2002benchmarking}. We briefly recall this concept: Let $\cS$ be the set of considered models, $\cK$ the set of instances and $t_{k,s}$ the runtime of model $s$ on instance $k$. We assume $t_{k,s}$ is set to infinity (or large enough) if model $s$ does not solve instance $k$ within the time limit. The percentage of instances for which the performance ratio of solver $s$ is within a factor $\tau \geq 1$ of the best ratio of all solvers is given by:
\[k_s (\tau) = \frac{1}{\lvert \cK \rvert} \; \Bigg\lvert \Bigg\{ k\in \cK \; \vert \; \frac{t_{k,s}}{\text{min}_{\hat{s}\in \cS} t_{k,\hat{s}}}  \leq \tau \Bigg\} \Bigg \lvert  \]
Hence, the function $k_s$ can be viewed as the distribution function for the performance ratio, which is plotted in a performance profile for each
model.

For all experiments of the selection problem, we use CPLEX version 12.8 on an Intel Xeon Gold 6154 CPU computer server running at 3.00GHz with 754 GB RAM. For the experiments of both the job assignment and 2-edge-connected subgraph problem we use CPLEX version 22.11 on an Intel pc-i440fx-7.2 CPU computer server running at 2.00GHz with 15 GB RAM. All processes are restricted to one thread.

\subsection{Selection}

\subsubsection{Setup}

We conduct three types of experiments on robust selection problems of the type
\[ \min \left\{ \sum_{i\in[n]} d_i x_i : \sum_{i\in[n]} x_i \ge p,\ \pmb{x}\in\{0,1\}^n \right\} \]
with budgeted interdiction uncertainty in the cardinality constraint.
In the first experiment, we compare the tightness of the lower bounds obtained by all five models when the size of $n$ and $p$ is fixed. The remaining experiments compare the solution times of the proposed models.

In the second experiment, we vary the size of $n$. This is done in two ways: once for a fixed value of $p$, and once for a value of $p$ that increases with $n$. In the third experiment, we fix $n$ and vary only the size of $p$. This way, we can thoroughly evaluate the effect these two parameters have on the performance of the models.

We use two techniques to generate instances for the selection problem under budgeted interdiction uncertainty, called \texttt{Gen-1} and \texttt{Gen-2}. In \texttt{Gen-1}, the weights are chosen independently from the corresponding costs, which meanst that there may be both particularly good items (with low $c_i$ and high $w_i$) and bad items. In \texttt{Gen-2}, the weight of items depend on their costs, which intuitively may lead to harder instances compared to \texttt{Gen-1}. The generation methods are considered as follows:
\begin{itemize}
\item \textbf{\texttt{Gen-1}:} for each $i\in[n]$ we choose $d_i, w_i$ from $\{1,\ldots,100\}$ independently random uniform
\item \textbf{\texttt{Gen-2}:} for each $i\in[n]$ the value of $w_i$ depends on the value of $d_i$, thus we choose $d_i$ from $\{1,\ldots,100\}$ and $w_i$ from $\{\max(1,d_i-5),\ldots,\min(100 , d_i+5)\}$ randomly uniform
\end{itemize}
In both cases, we set
\[B = \Bigg\lfloor \frac{\sum_{i\in[n]} w_i}{4} \Bigg\rfloor \]

\subsubsection{Experiment 1}
\label{subsec:experiment-1}

Here we focus on the LP-relaxation of all models and compare the lower bounds obtained by each of them. In this experiment we fix $n=10$ and $p=1$ to include all models. We solve the LP relaxations of 1000 instances for each combination of generation and solution methods using CPLEX. We then perform a pairwise comparison of the resulting lower bounds.

The results of this experiment is presented in Tables~\ref{table-lb-comparison-Gen-1} and \ref{table-lb-comparison-Gen-2} for \texttt{Gen-1} and \texttt{Gen-2}, respectively. Each number shows how many times the method in the respective row provided a strictly better (in this case higher) lower bound than the model in the correspondence column. The last column shows the average of cases per row where the model has been better then the comparison model in percent.
\begin{table}[htbp]
	\begin{center}
		\begin{tabular}{c|ccccc|c}
			      & IP-1 & IP-2 & IP-3 & IP-4 & IP-5 & $\%$\\
			\hline
			IP-1 & \textemdash & 979 & 967 & 993 & 983 & 98.05 \\
			IP-2 & 21 & \textemdash & 5 & 368 & 92 & 12.15\\
			IP-3 & 33 & 982 & \textemdash & 878 & 412 & 57.63\\
			IP-4 & 7 & 628 & 118 & \textemdash & 9 & 19.05\\
			IP-5 & 17 & 898 & 575 & 987 & \textemdash & 61.93
		\end{tabular}
		\caption{LB comparison (\texttt{Gen-1}).}\label{table-lb-comparison-Gen-1}
	\end{center}
\end{table}

Based on the information provided in Table~\ref{table-lb-comparison-Gen-1} for \texttt{Gen-1}, we note that IP-1 dominates all other models in over 98 percent of cases (which is not surprising, as it is the most specialized model). The next best model is IP-5, followed by IP-3. With some gap behind these two models follow IP-4 and IP-2. The weakest model, IP-2, is stronger than another model in only around 12 percent of cases.
 
\begin{table}[htbp]
	\begin{center}
		\begin{tabular}{c|ccccc|c}
			& IP-1 & IP-2 & IP-3 & IP-4 & IP-5 & $\%$\\
			\hline
			IP-1 & \textemdash & 1000 & 999 & 1000 & 1000 & 99.98\\
			IP-2 & 0 & \textemdash & 0 & 999 & 586 & 39.63\\
			IP-3 & 1 & 1000 & \textemdash & 1000 & 1000 & 75.03\\
			IP-4 & 0 & 1 & 0 & \textemdash & 0 & \phantom{0}0.03\\
			IP-5 & 0 & 414 & 0 & 1000 & \textemdash & 35.35
		\end{tabular}
		\caption{LB comparison (\texttt{Gen-2}).}\label{table-lb-comparison-Gen-2}
	\end{center}
\end{table}

Interestingly, this ordering changes when using instances of type \texttt{Gen-2}, see Table~\ref{table-lb-comparison-Gen-2}. While IP-1 still outperforms other models other models in nearly all cases (over 99 percent), the second best model is IP-3, which performs relatively better than before. Similarly, IP-2 has improved in the ranking, while IP-5 (which was the second best choice in Table~\ref{table-lb-comparison-Gen-1} is now relegated to fourth place. IP-4 can provide a better bound than another model in only one single instance.

\subsubsection{Experiment 2}
\label{subsec:experiment-2}

In this experiment, we vary the problem size in $n\in\{20,25,\ldots,100\}$.
The experiment is divided into two parts. In the first part, we fix $p$ to $1$ so that all solution methods could be included, and also consider $p=5$ to compare the performance of solution methods which can be applied to cases where $p>1$. In the second part, we use $p=\frac{n}{5}$ so that $p$ grows linearly in $n$. For each combination of generation and solution methods we solved 50 instances using CPLEX and with a 600 second time limit. We always present a plot of average solution times and a performance profile.

\begin{figure}[htbp]
	\begin{center}
		\subfigure[\texttt{Gen-1}]{\includegraphics[width=0.45\textwidth]{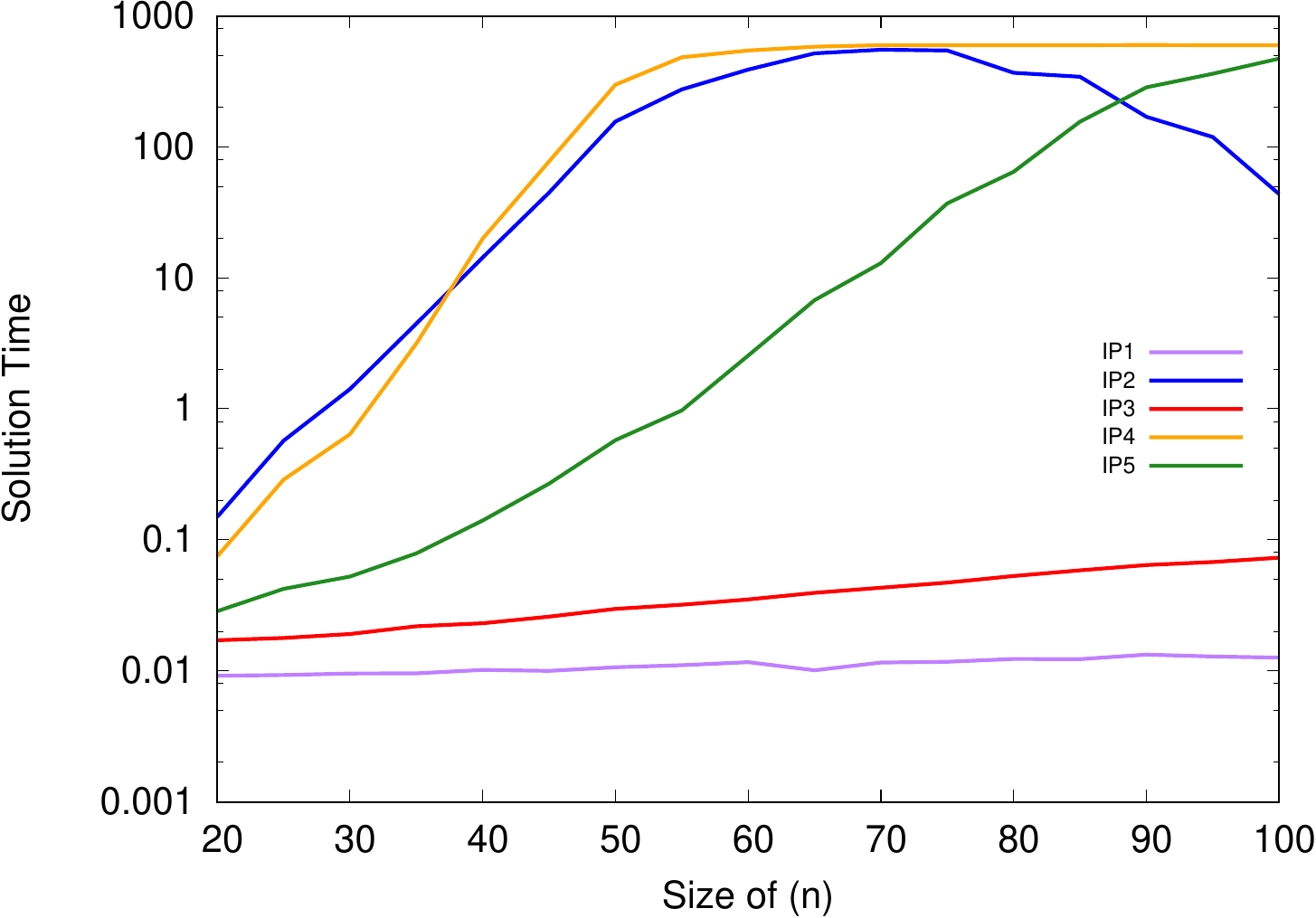}}
		\subfigure[\texttt{Gen-2}]{\includegraphics[width=0.45\textwidth]{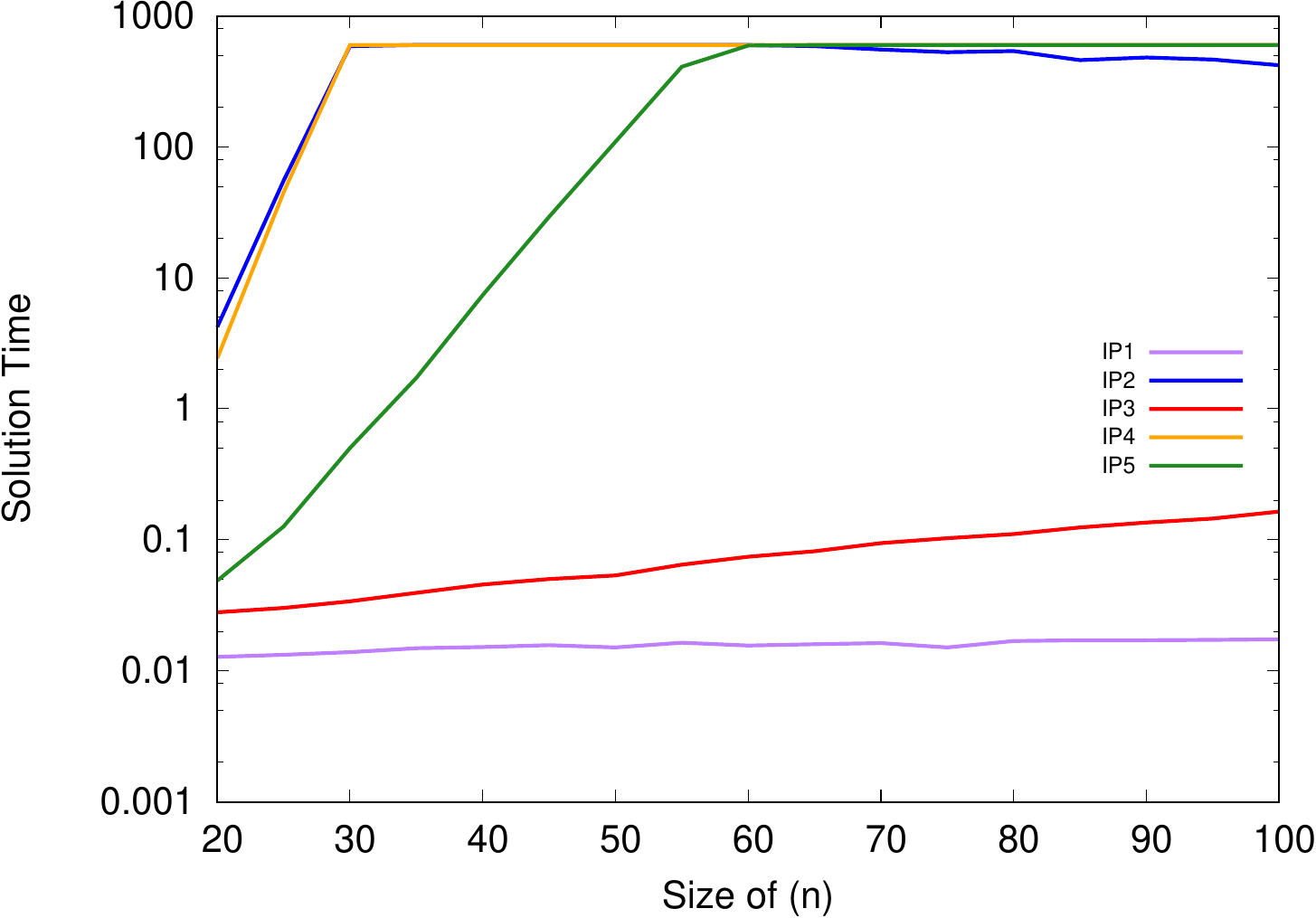}}
	\end{center}
	\caption{Selection - Exp2 - Solution times ($p=1$)}\label{exp2-part1-p1-time}
\end{figure}

\begin{figure}[htbp]
	\begin{center}
		\subfigure[\texttt{Gen-1}]{\includegraphics[width=0.45\textwidth]{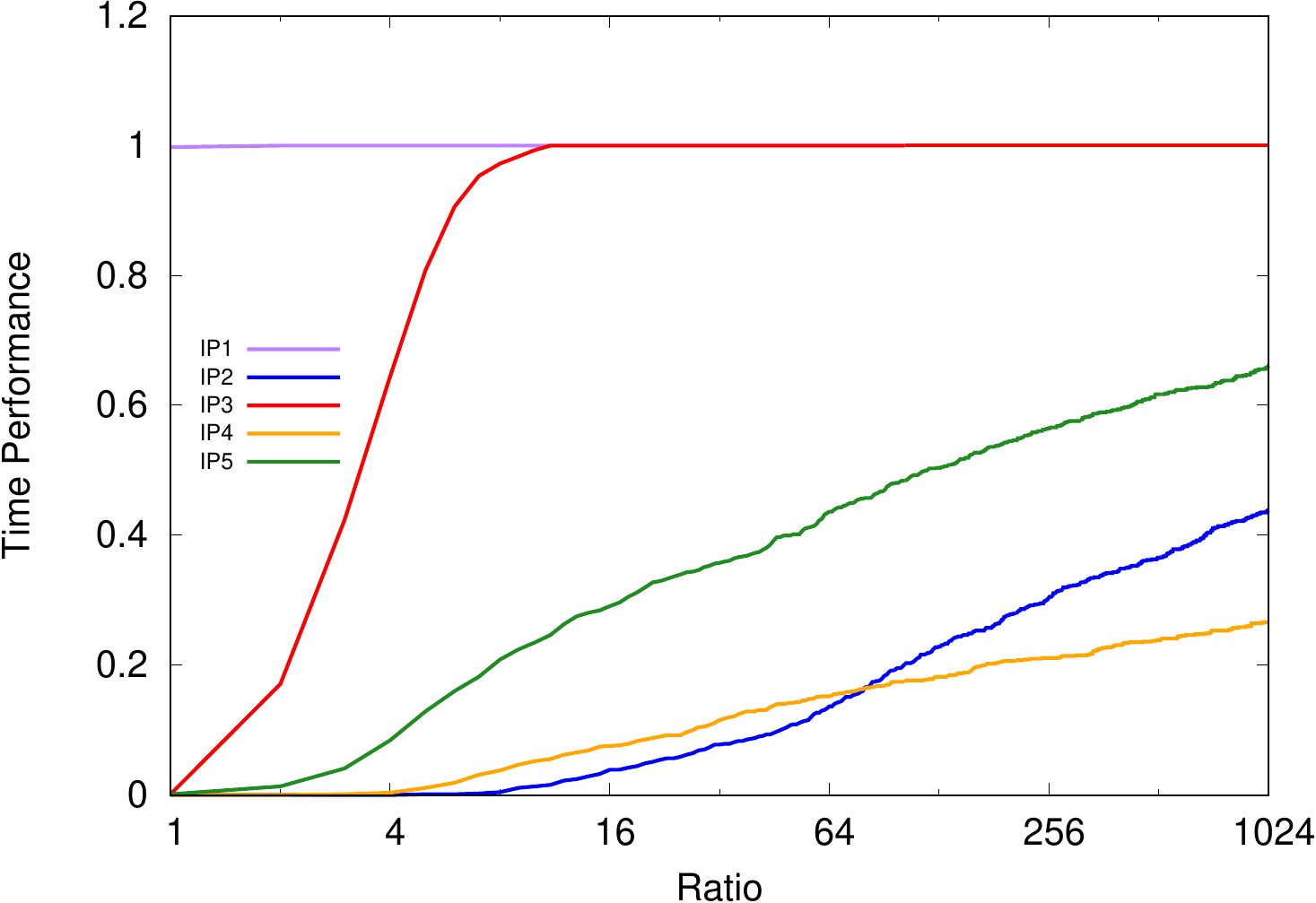}}
		\subfigure[\texttt{Gen-2}]{\includegraphics[width=0.45\textwidth]{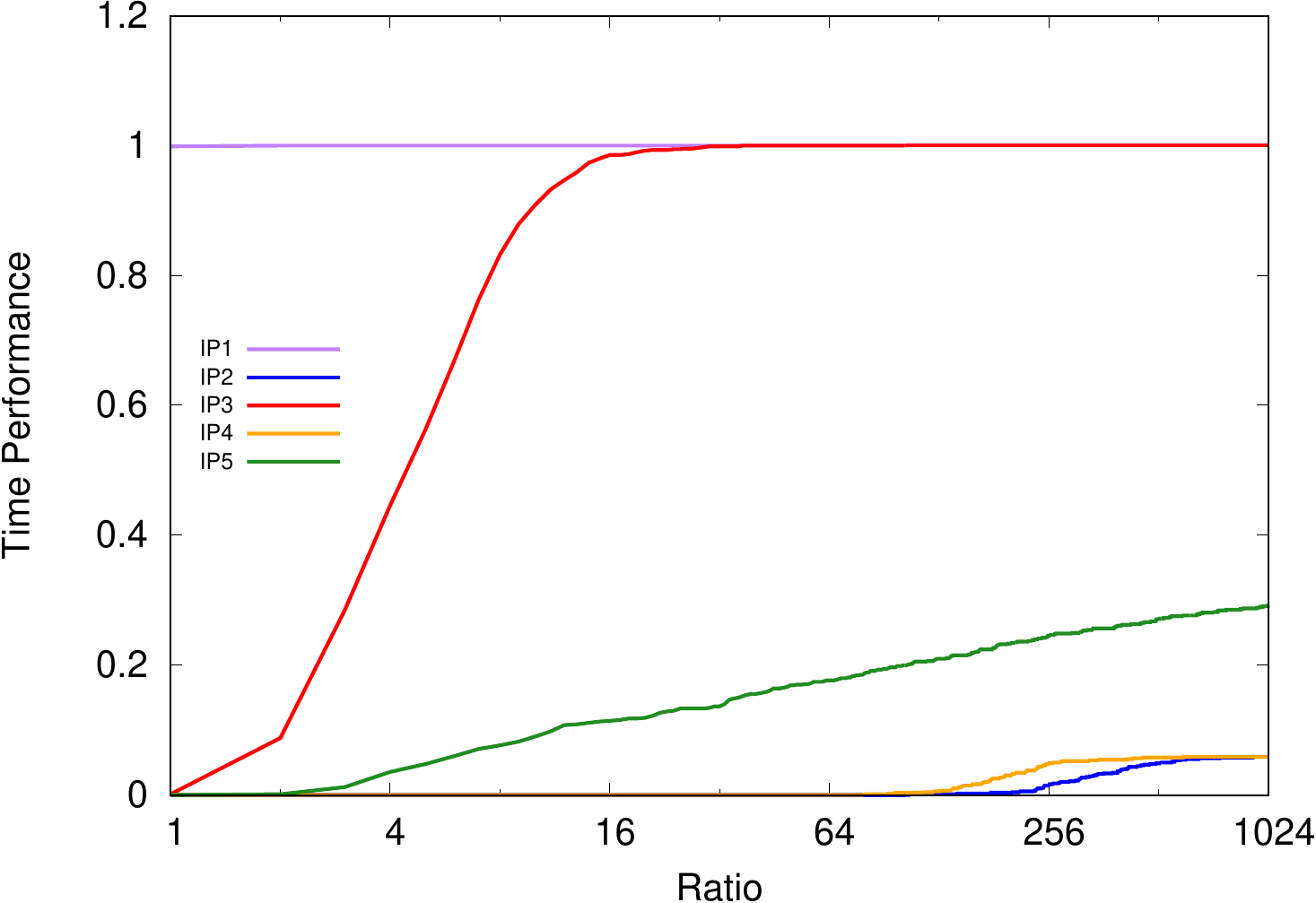}}
	\end{center}
	\caption{Selection - Exp2 - Performance profiles ($p=1$)}\label{exp2-part1-p1-pp}
\end{figure}

Figure~\ref{exp2-part1-p1-time} shows the solution times for $p=1$. Clearly, IP-1 is the fastest model to solve instances for both \texttt{Gen-1} and \texttt{Gen-2}, followed by IP-3. Other models also show similar behavior for both generation methods. Interestingly, IP-2 may even become faster as $n$ increases, which seems counterintuitive, but can be explained by the fact that $p$ remains constant. The performance profiles (see Figure~\ref{exp2-part1-p1-pp}) reflect a similar relative performance of the five models.

\begin{figure}[htbp]
	\begin{center}
		\subfigure[\texttt{Gen-1}]{\includegraphics[width=0.45\textwidth]{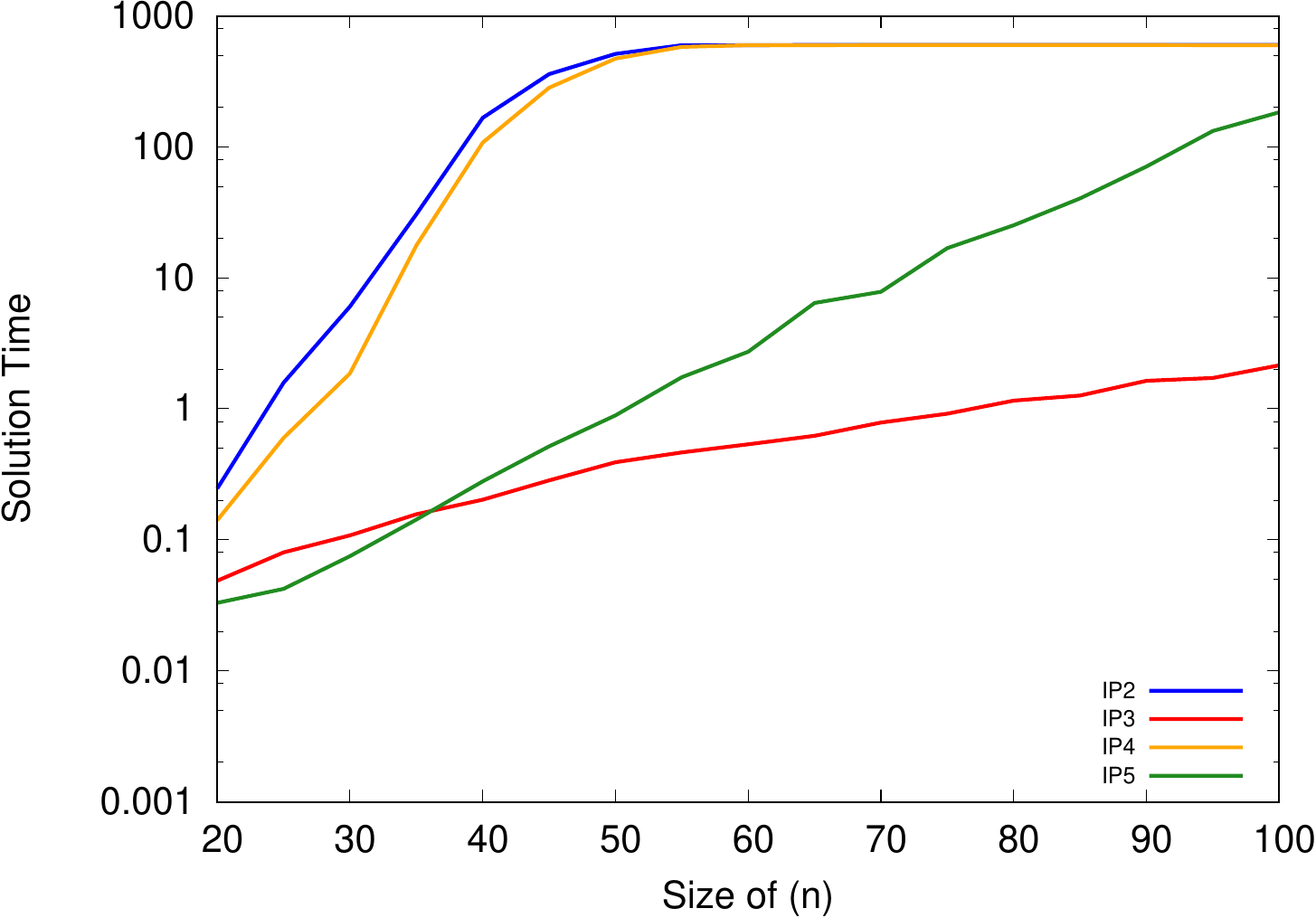}}
		\subfigure[\texttt{Gen-2}]{\includegraphics[width=0.45\textwidth]{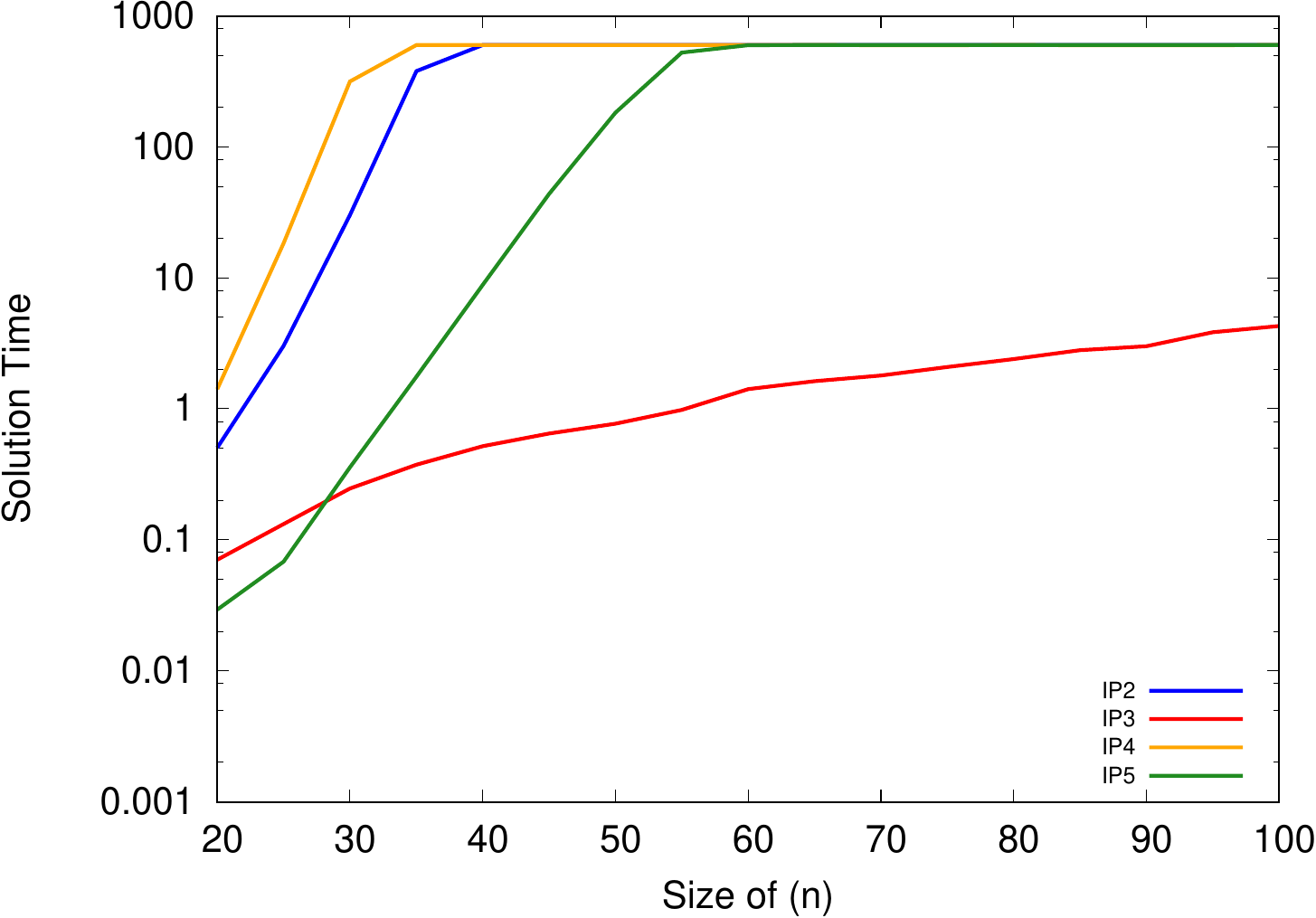}}
	\end{center}
	\caption{Selection - Exp2 - Solution times ($p=5$)}\label{exp2-part1-p5-time}
\end{figure}

\begin{figure}[htbp]
	\begin{center}
		\subfigure[\texttt{Gen-1}]{\includegraphics[width=0.45\textwidth]{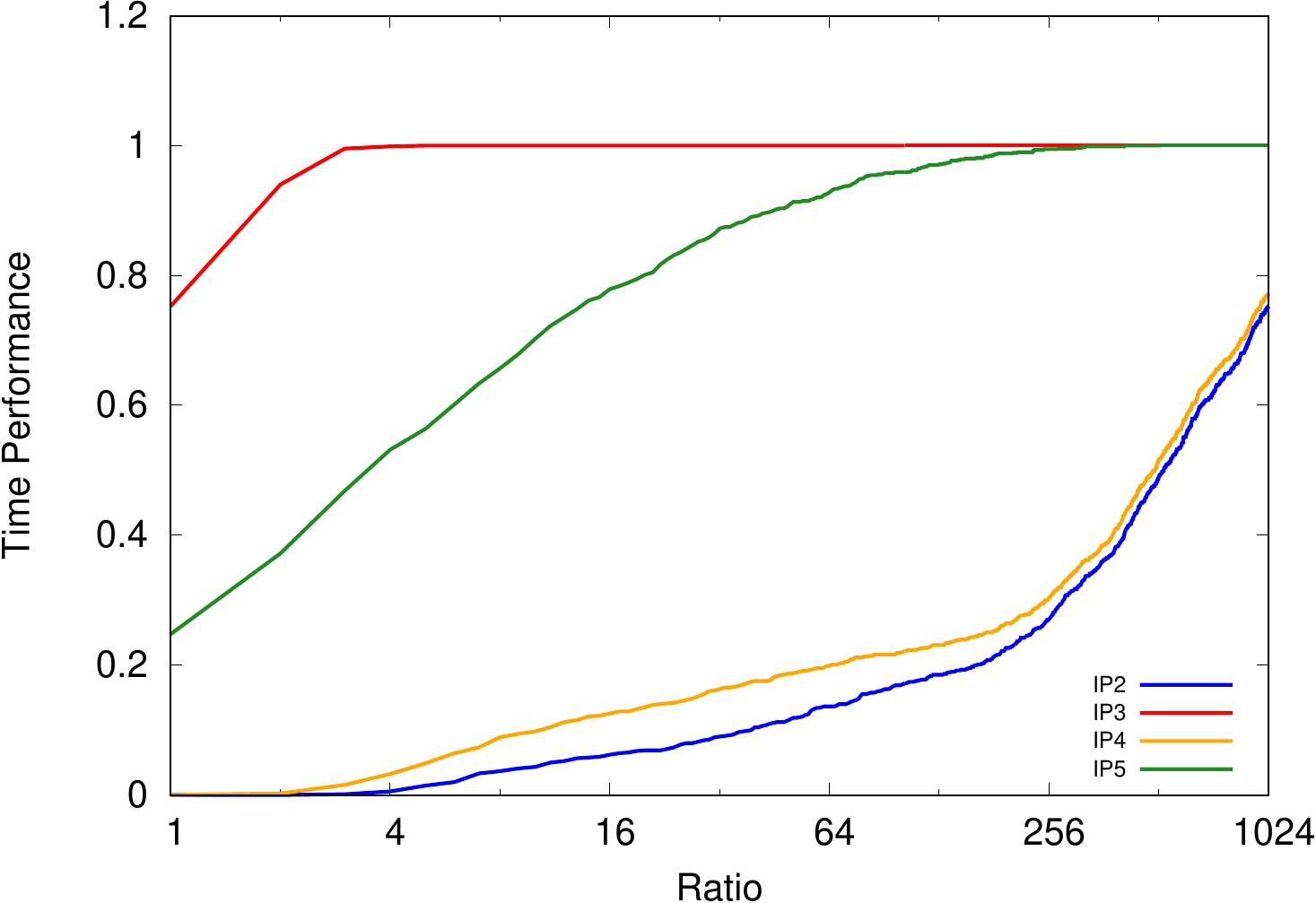}}
		\subfigure[\texttt{Gen-2}]{\includegraphics[width=0.45\textwidth]{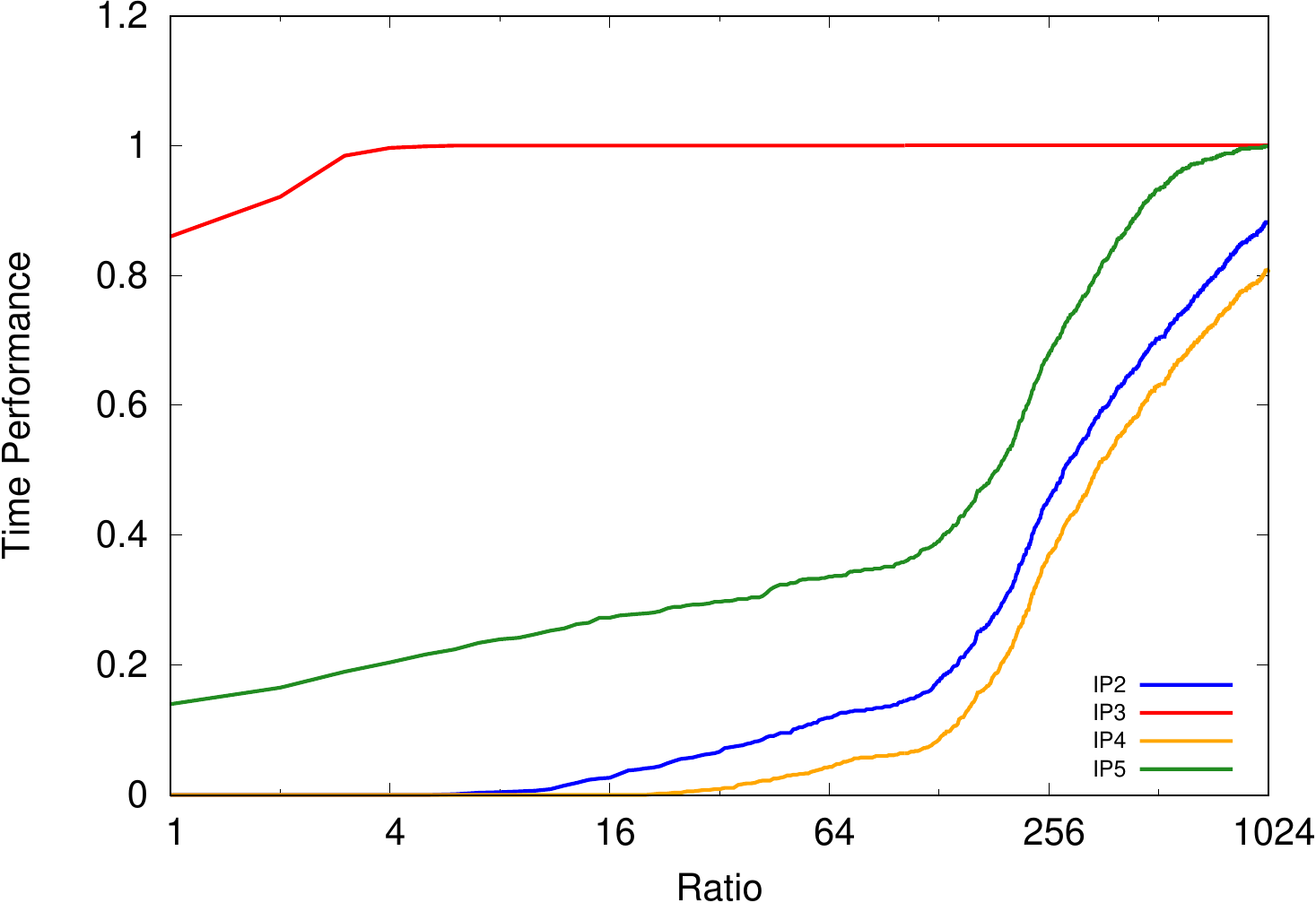}}
	\end{center}
	\caption{Selection - Exp2 - Performance profiles ($p=5$)}\label{exp2-part1-p5-pp}
\end{figure}

In Figures~\ref{exp2-part1-p5-time} and \ref{exp2-part1-p5-pp}, we show the average solution times and performance profiles for the case $p=5$, where IP-1 is not included. A similar behavior to the cases when $p=1$ can be seen. As IP-1 is excluded, here IP-3 has the best average solution time. The difference is that IP-3 fails to be faster than IP-5 for instances with smaller size of $n$. Similarly, in the performance profile, IP-3 dominates other models. The main difference compared to the case with $p=1$ is that none of the models is always superior to others.
 
We now consider the case $p=\frac{n}{5}$, i.e., $p$ grows linearly with $n$. The average solution times and performance profiless of this experiment are presented in Figures~\ref{exp2-part2-time} and \ref{exp2-part2-pp}, respectively.

\begin{figure}[htbp]
	\begin{center}
		\subfigure[\texttt{Gen-1}]{\includegraphics[width=0.45\textwidth]{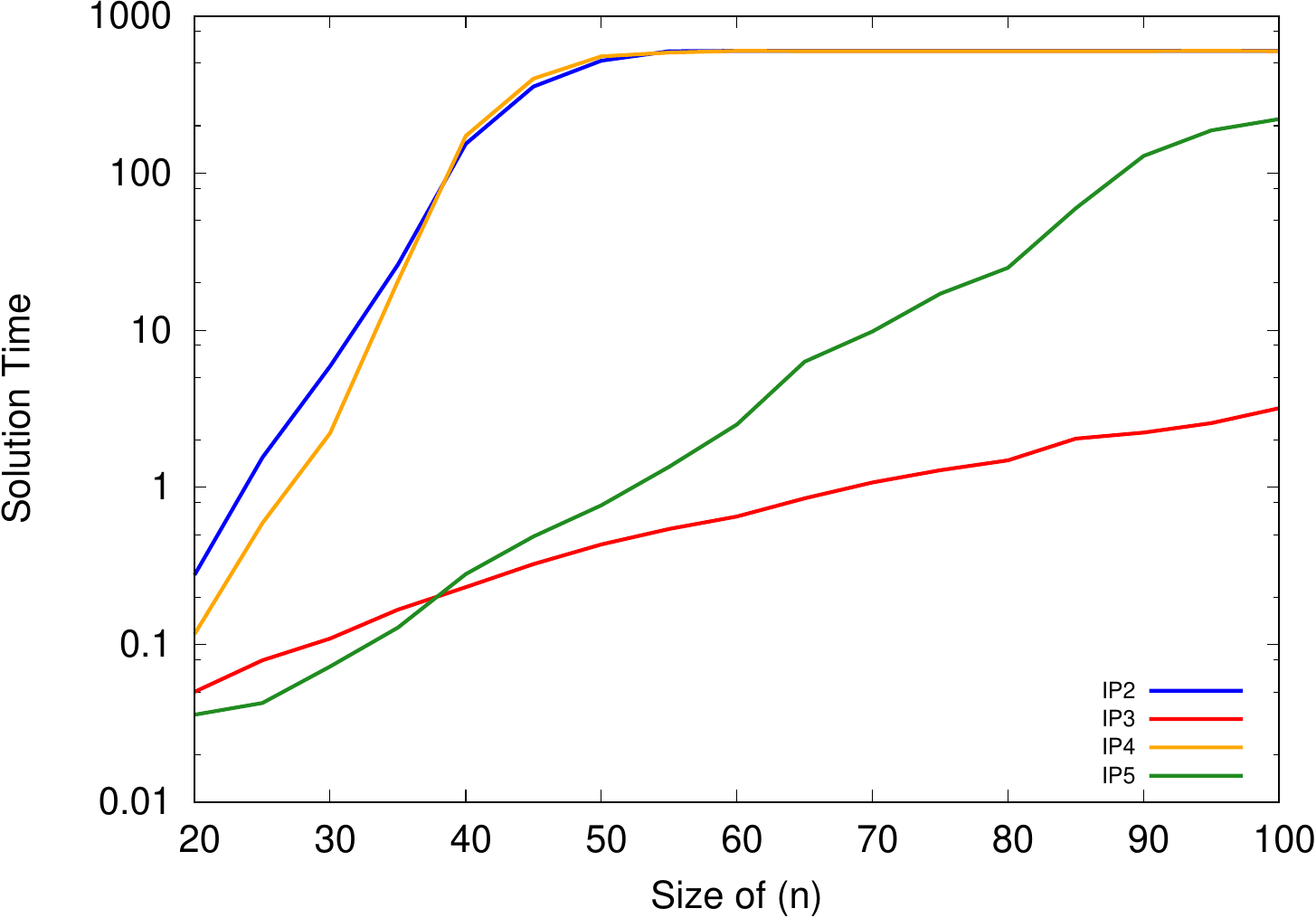}}
		\subfigure[\texttt{Gen-2}]{\includegraphics[width=0.45\textwidth]{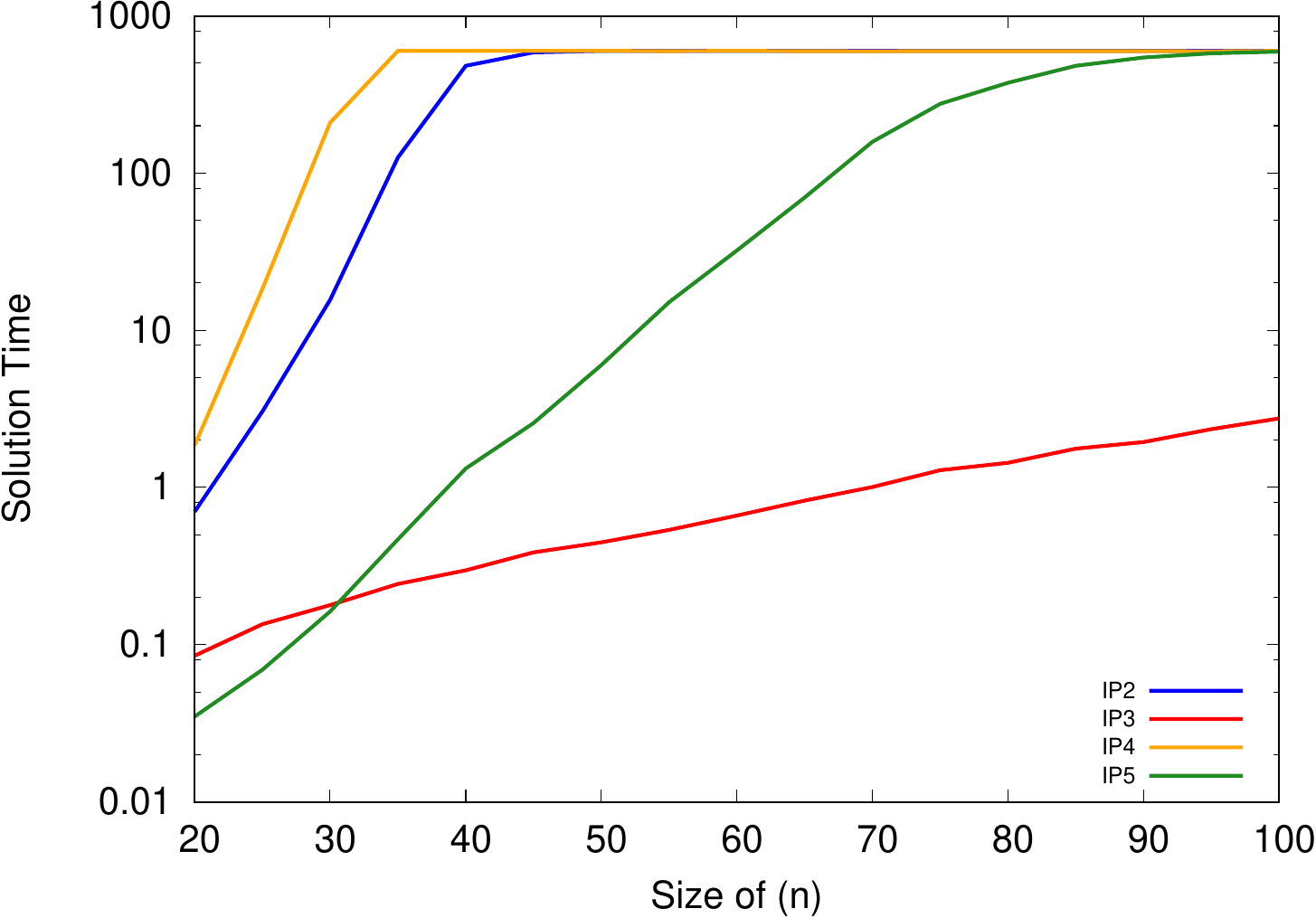}}
	\end{center}
	\caption{Selection - Exp2 - Solution times ($p=n/5$)}\label{exp2-part2-time}
\end{figure}

\begin{figure}[htbp]
	\begin{center}
		\subfigure[\texttt{Gen-1}]{\includegraphics[width=0.45\textwidth]{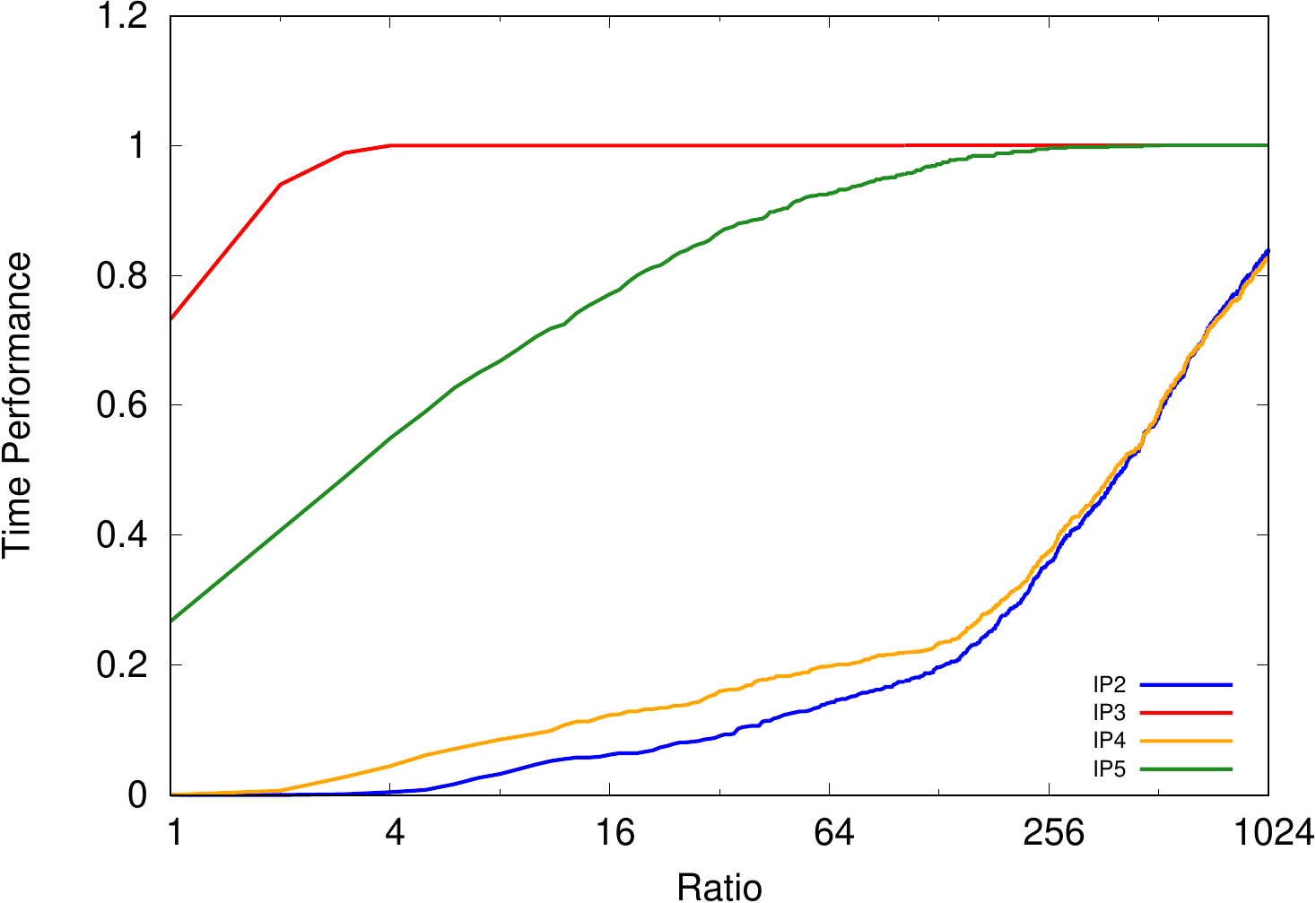}}
		\subfigure[\texttt{Gen-2}]{\includegraphics[width=0.45\textwidth]{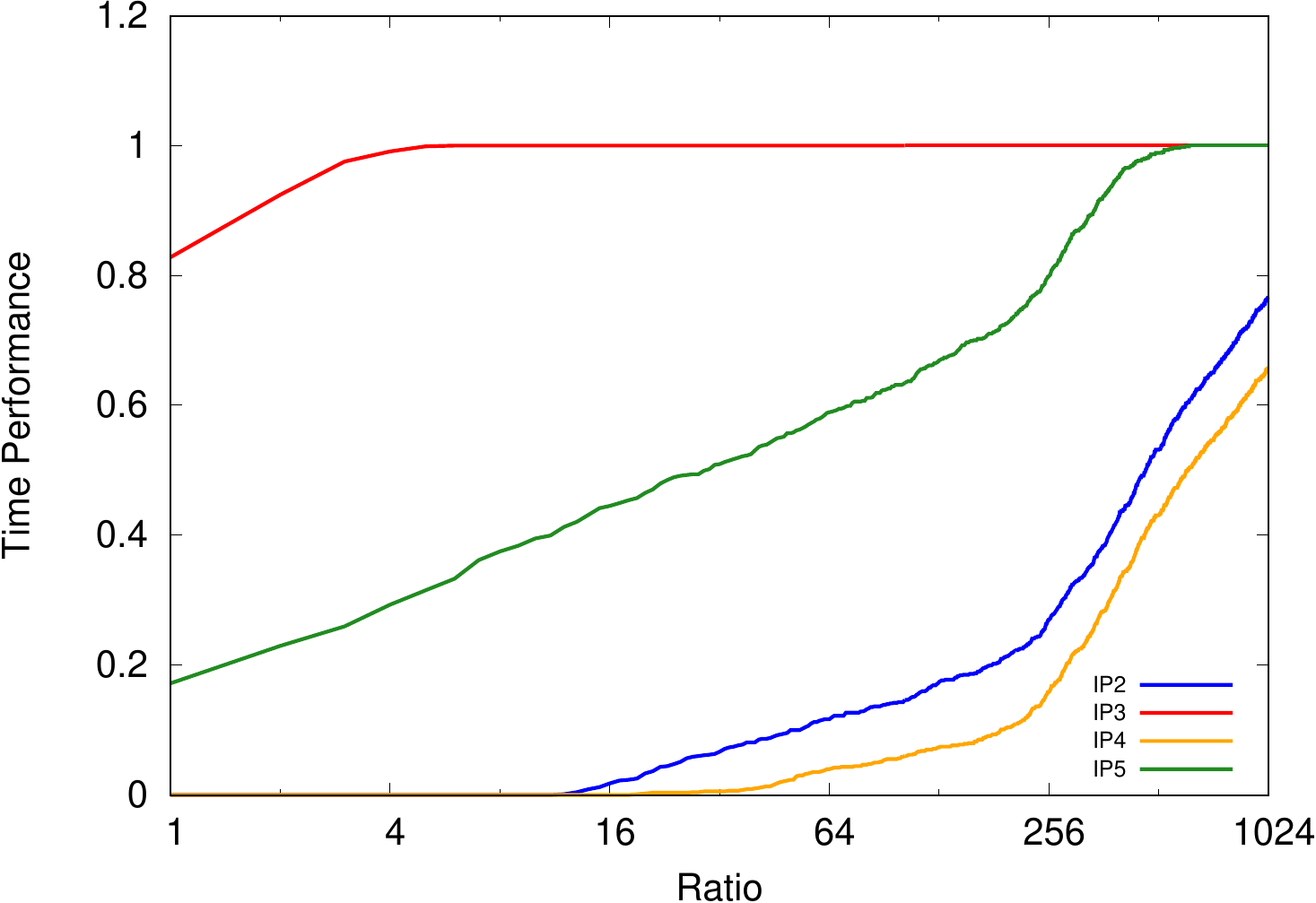}}
	\end{center}
	\caption{Selection - Exp2 - Performance profiles ($p=n/5$)}\label{exp2-part2-pp}
\end{figure}

The results of the second part is closely similar to the first part results when $p=5$. That is, IP-3 beats all formulations except IP-5 with $n=20,30,40$ for \texttt{Gen-1} and $n=20,30$ for \texttt{Gen-2}. In addition, the behavior of IP-2 and IP-4 is similar, however their comparison is difficult because of the time limit. In this experiment, IP-4 has better solution time than IP-2 for \texttt{Gen-1}, but for \texttt{Gen-2}, IP-2 is slightly faster than IP-4.

\subsubsection{Experiment 3}
\label{subsec:experiment-3}

In this experiment, we fix the number of items $n$ and change the number of items we want to select $p$. We consider the cases $n=20$ and $n=40$. In both cases, $p$ is chosen from $\{1,2,\ldots,10\}$. As before, we solved 50 instances using CPLEX with a 600-second time limit for each combination of generation and solution methods. The results of this experiment is provided in Figures~\ref{exp3-n20-time}-\ref{exp3-n40-pp}.

\begin{figure}[htbp]
	\begin{center}
		\subfigure[\texttt{Gen-1}]{\includegraphics[width=0.45\textwidth]{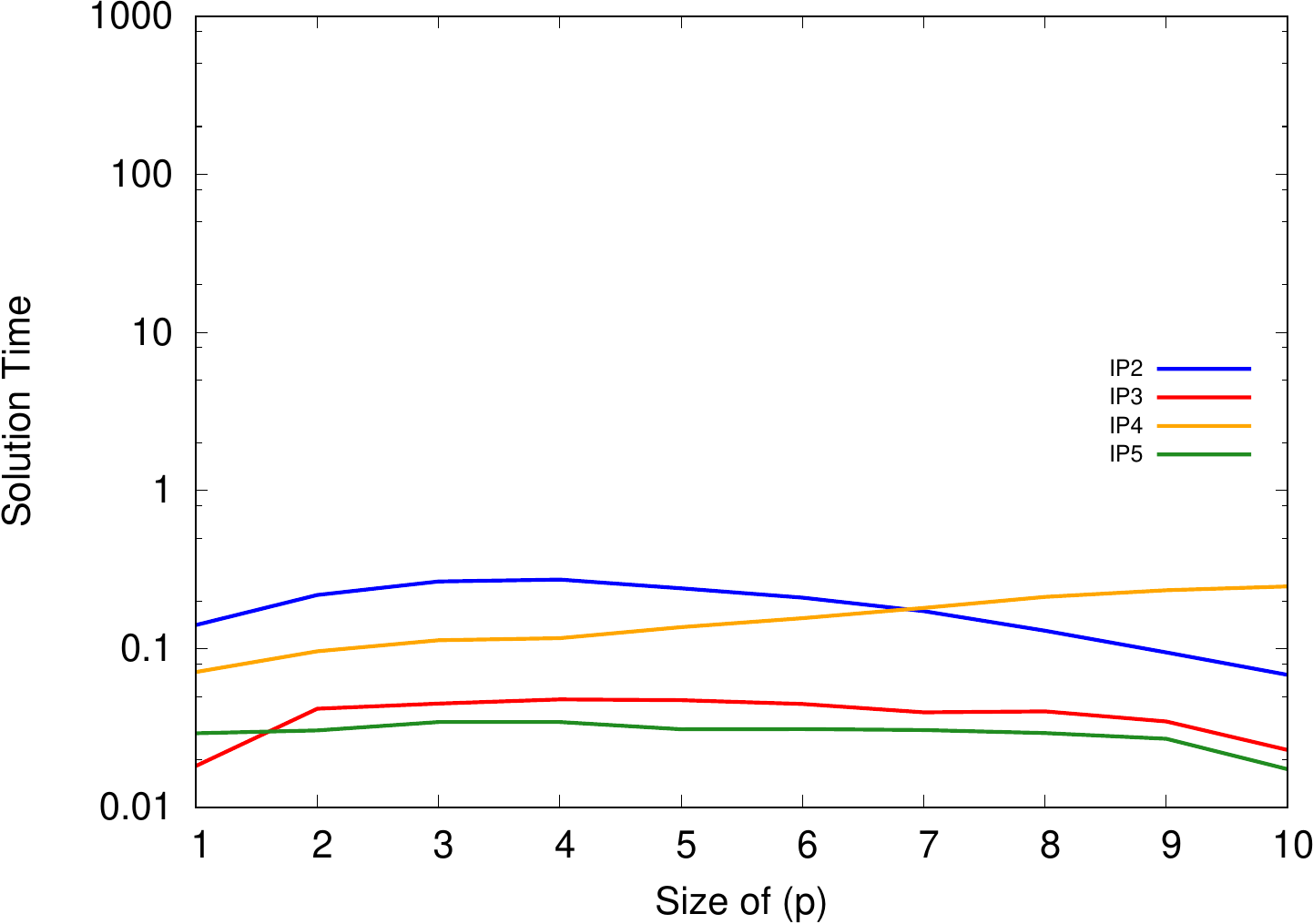}}
		\subfigure[\texttt{Gen-2}]{\includegraphics[width=0.45\textwidth]{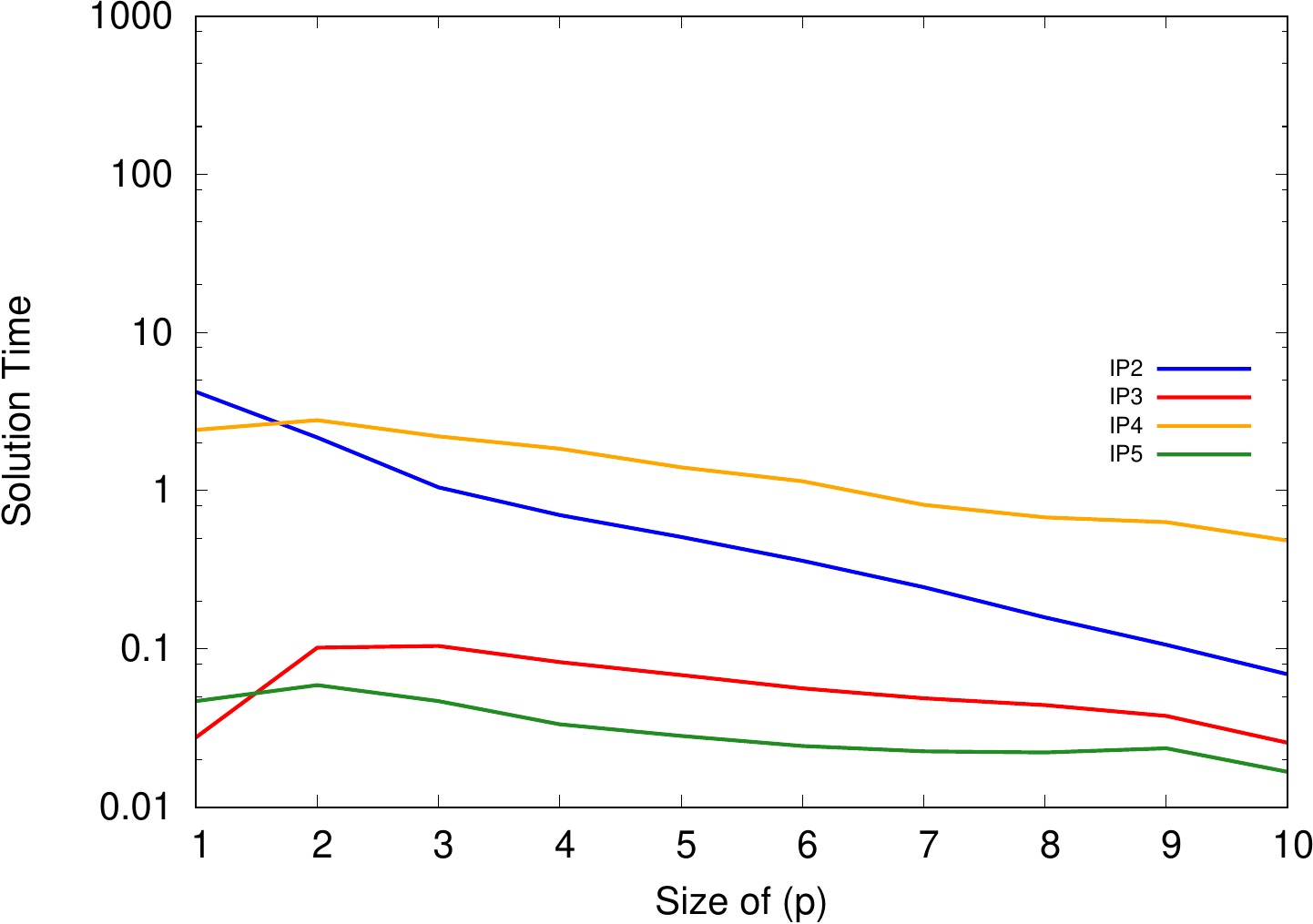}}
	\end{center}
	\caption{Selection - Exp3 - Solution times ($n=20$)}\label{exp3-n20-time}
\end{figure}

\begin{figure}[htbp]
	\begin{center}
		\subfigure[\texttt{Gen-1}]{\includegraphics[width=0.45\textwidth]{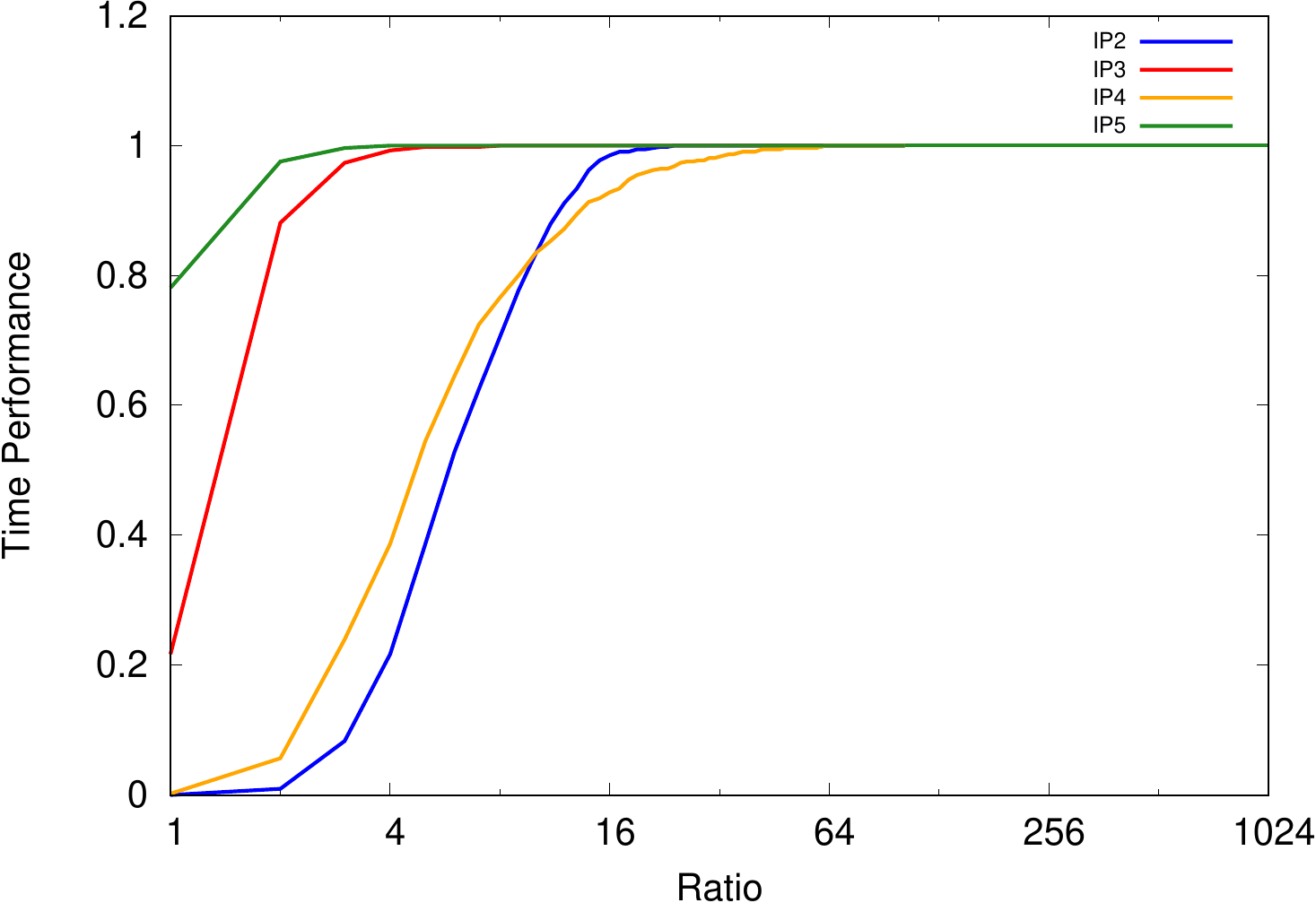}}
		\subfigure[\texttt{Gen-2}]{\includegraphics[width=0.45\textwidth]{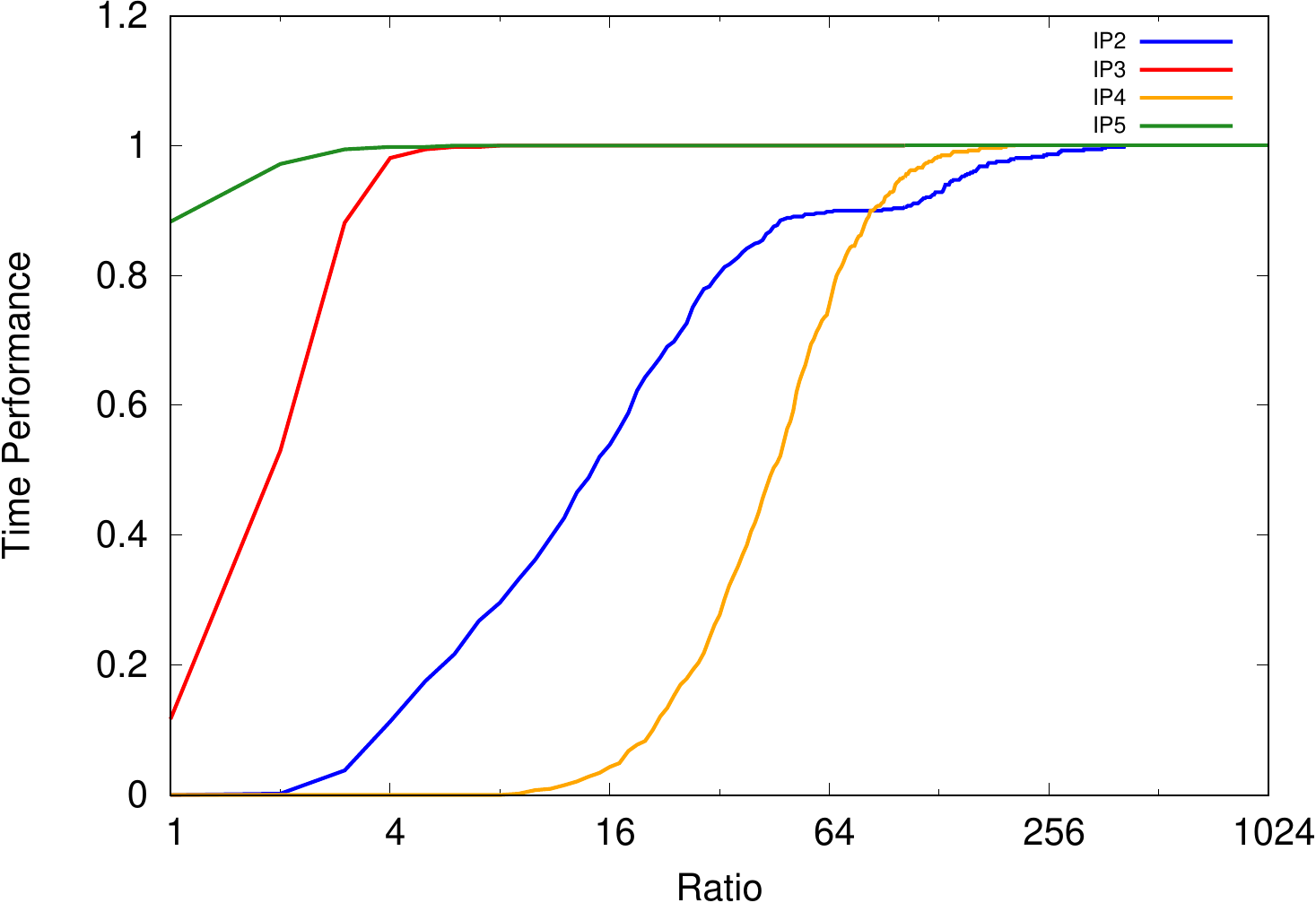}}
	\end{center}
	\caption{Selection - Exp3 - Performance profiles ($n=20$)}\label{exp3-n20-pp}
\end{figure}

Interestingly, the experiment for instances with $n=20$ shows that IP-2, IP-3 and IP-4 are dominated by IP-5 except for $p=1$ for both \texttt{Gen-1} and \texttt{Gen-2}. In this single case, IP-5 is outperformed by IP-3. This experiment shows that the problem solved faster for larger value of $p$. This can be explained by the observation that for larger values of $p$, nearly all items need to be selected (recall that we need to pack more than $p$ items to respect the uncertainty). Similar to other experiments, the performance profile (see Figure~\ref{exp3-n20-pp}) represents the obtained results for the average solution times also holds for the instance-wise comparisons of the given models.

\begin{figure}[htbp]
	\begin{center}
		\subfigure[\texttt{Gen-1}]{\includegraphics[width=0.45\textwidth]{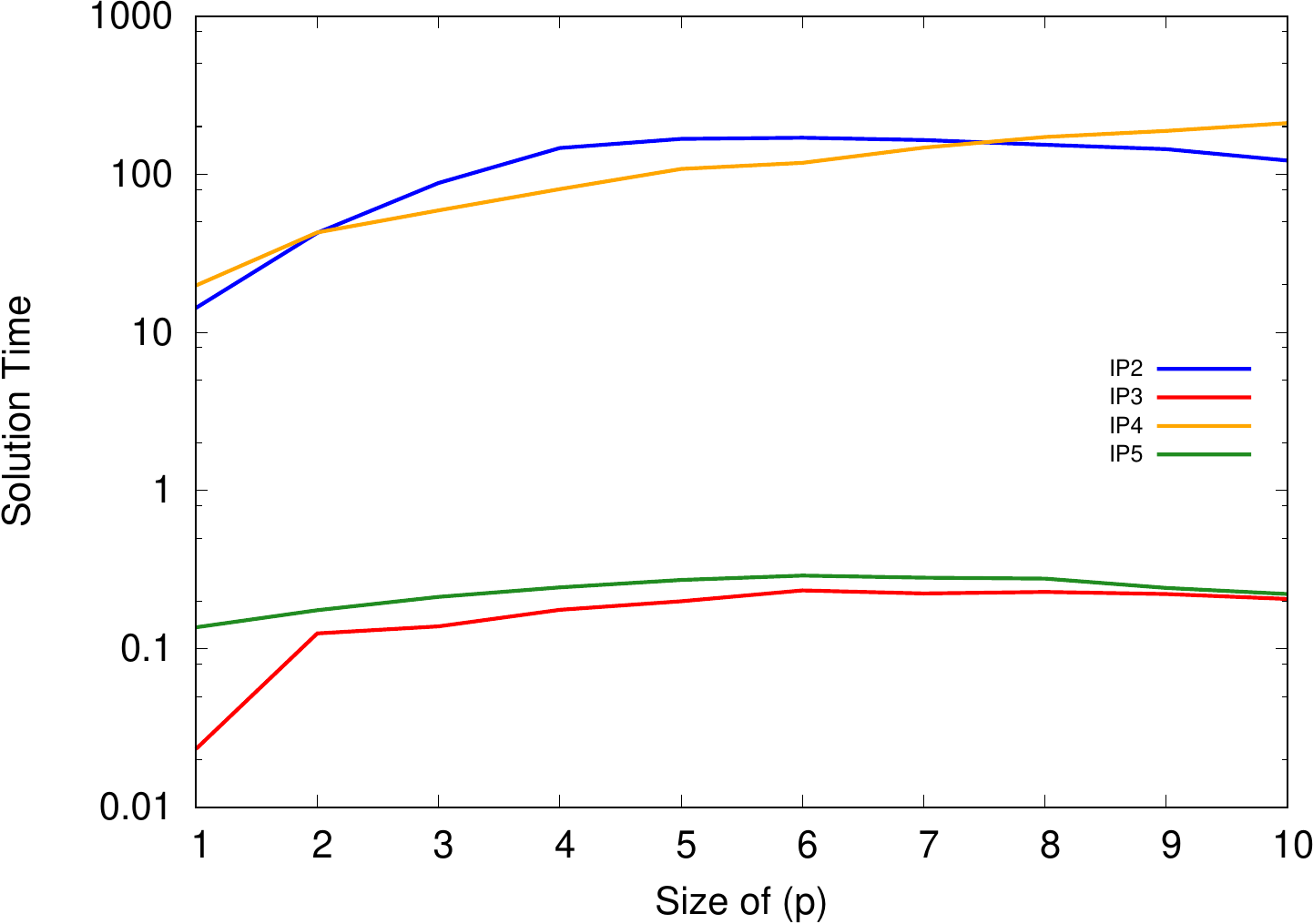}}
		\subfigure[\texttt{Gen-2}]{\includegraphics[width=0.45\textwidth]{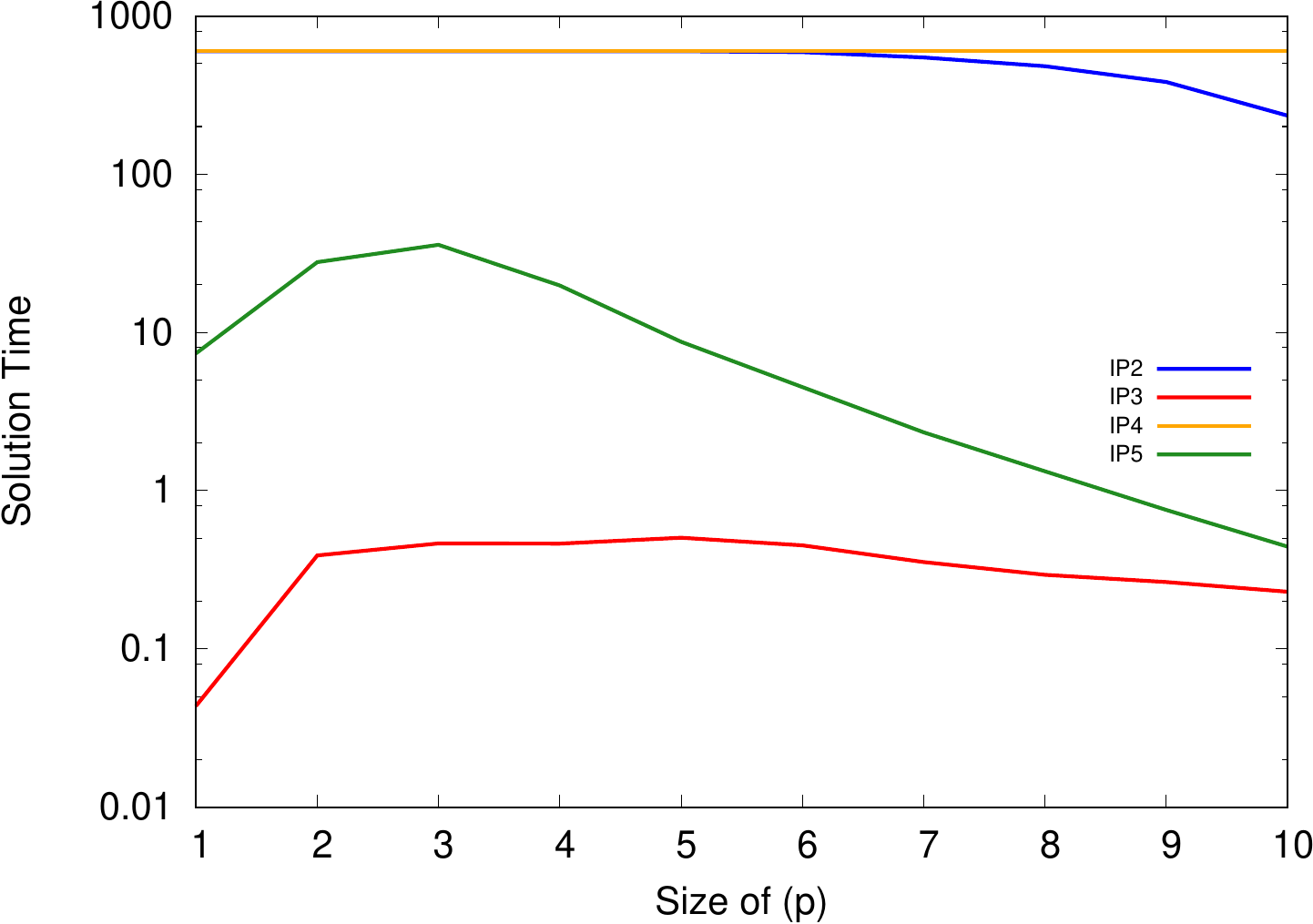}}
	\end{center}
	\caption{Exp3, solution times for $n=40$}\label{exp3-n40-time}
\end{figure}

\begin{figure}[htbp]
	\begin{center}
		\subfigure[\texttt{Gen-1}]{\includegraphics[width=0.45\textwidth]{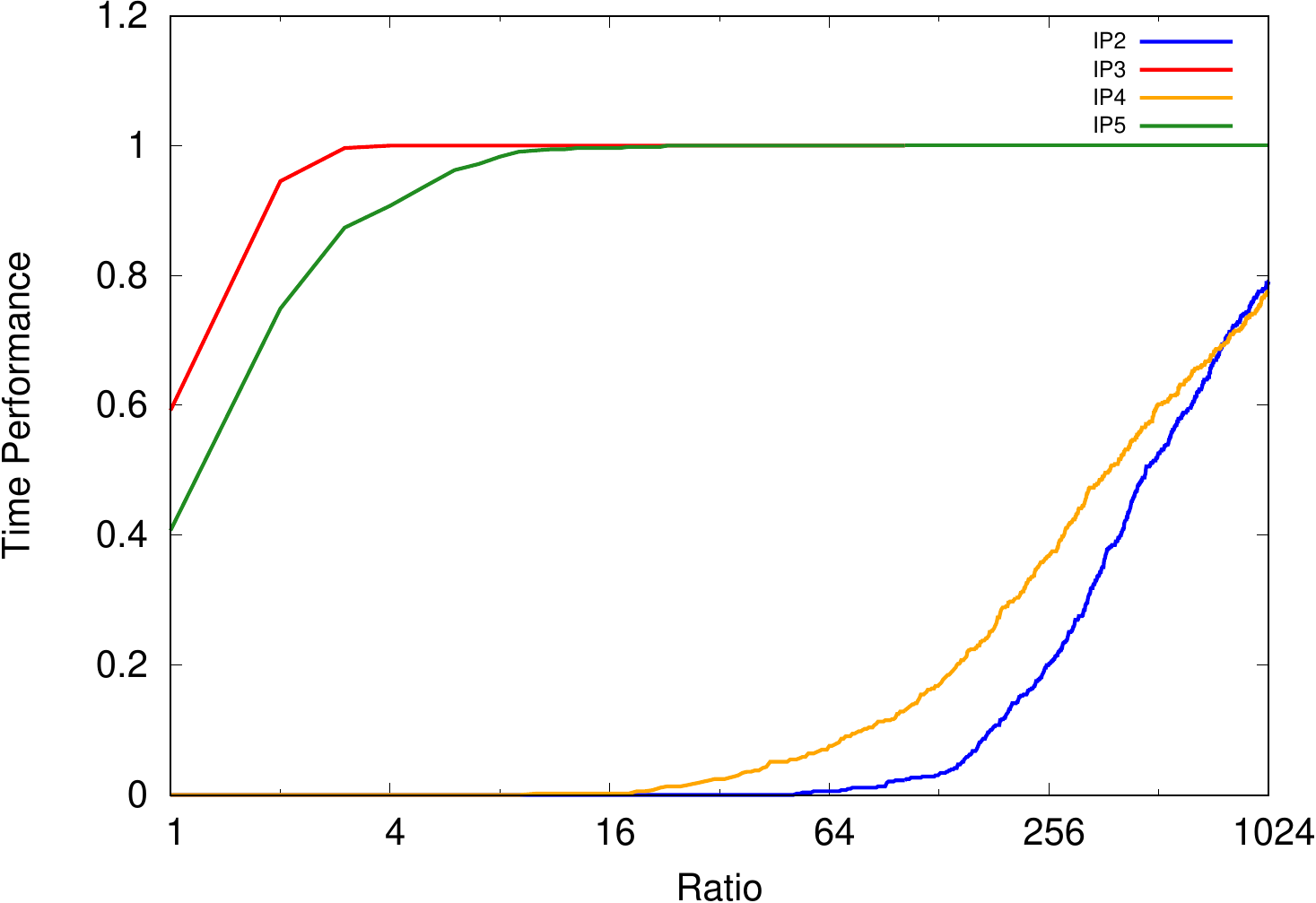}}
		\subfigure[\texttt{Gen-2}]{\includegraphics[width=0.45\textwidth]{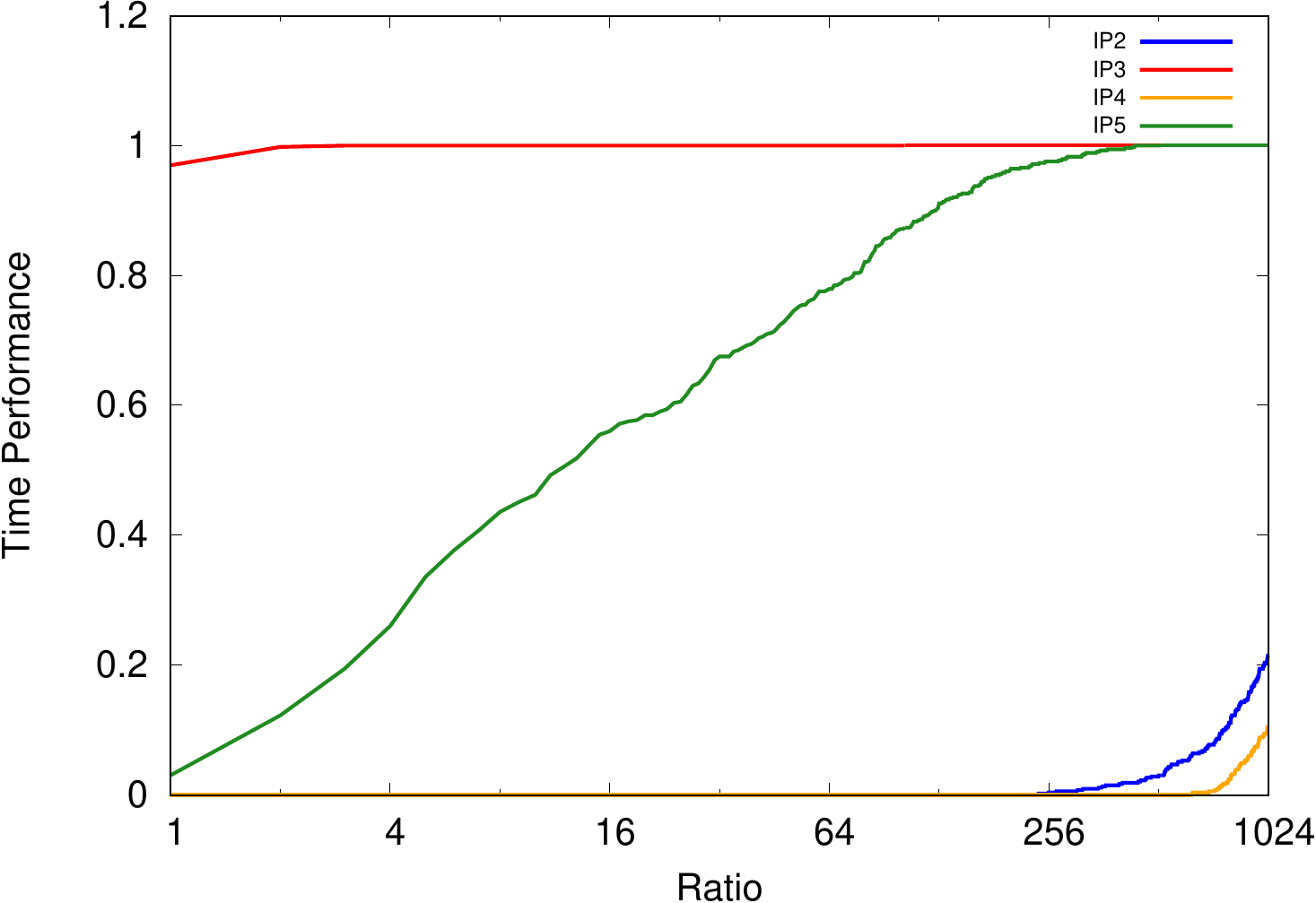}}
	\end{center}
	\caption{Selection - Exp3 - Performance profiles ($n=40$)}\label{exp3-n40-pp}
\end{figure}

The results depicted in Figures~\ref{exp3-n40-time} and \ref{exp3-n40-pp} show that the problems first tend to become harder to solve and then the solution time falls. Notably, unlike the case with $n=20$, IP-3 has the best performance with regard to the solution time. Another difference is that IP-3 is always superior in comparison to the other mathematical formulations. In this case the problem in considerably harder to solve than cases with $n=20$. In this sense, the problem hits the time limit even for $p=1$, while IP-3 never reaches even close to the time limit. The performance profile shows that IP-3 is almost always faster than other IPs.

\subsection{Job Assignment}

\subsubsection{Setup}

We now consider a job assignment problem with $m$ jobs and $n$ workers. Each job has a profit $p_j$ and workers demand of $d_j$ for all $j\in[m]$. Here, instead of weights for each item in the selection problem, we have a failure probability for each worker $w_i$ for all $i\in[n]$. Therefore, the robust job assignment problem under the budgeted interdiction uncertainty can be formulated as follows
\begin{align*}
\max\ &\sum_{j\in[m]} p_j \bar{z}_j \\
\text{s.t. } & \sum_{i\in[n]} (1-c_i) x_{ij} \ge d_j \bar{z}_j & \forall j\in[m], \pmb{c}\in \cU \\
& \sum_{j\in[m]} x_{ij} \le 1 & \forall i\in[n] \\
& x_{ij},\bar{z}_j\in\{0,1\}
\end{align*}

In this case and in order to see the performance of our IPs over the job assignment problem, we only consider IP-2, IP-3, IP-4 and IP-5. The compact formulations of these four IPs are collected in Appendix~\ref{app:jobassignment}. Furthermore, we introduce two experiments, where in experiment~1 we fix the number of jobs and change the number of workers; while in experiment~2 we fix the number of workers and vary the number of jobs.

Similar to the experiments on the selection problem, we use two types of instance generation methods for the job assignment problem under budgeted interdiction uncertainty, called \texttt{Gen-1} and \texttt{Gen-2}. In \texttt{Gen-1}, the job demands are chosen independently from the corresponding profits. In \texttt{Gen-2}, however, the profits of jobs depend on their demands. The generation methods are considered as follows:

\textbf{\texttt{Gen-1}}
\begin{itemize}
\item for each $i\in[m]$ we choose $d_i$ from $\{1,\ldots,\frac{2n}{m}\}$ independently random uniform.
\item for each $i\in[m]$ we choose $p_i$ from $\{1,\ldots,25\}$ independently random uniform.
\end{itemize}
\textbf{\texttt{Gen-2}}
\begin{itemize}
\item for each $i\in[m]$ we choose $d_i$ from $\{1,\ldots,\frac{2n}{m}\}$ independently random uniform.
\item for each $i\in[m]$ the value of $p_i$ depends on the value of $d_i$. If $d_i\le\frac{n}{m}$ then $p_i$ is chosen from $\{1,\ldots,25\}$, otherwise we choose $p_i$ from $\{10,\ldots,34\}$ randomly uniform. 
\end{itemize}
In both cases, we choose $w_i$ from $\{101,\ldots,150\}$ randomly uniform for all $i \in[n]$. Then we set
\[B = \Bigg\lfloor \frac{\sum_{i\in[n]} 2 w_i}{n} \Bigg\rfloor \]

\subsubsection{Experiment 1}\label{subsubsec:job-exp1}

In this experiment, we fix the number of jobs ($m$) and change the number of workers ($n$). To this end, we consider the case when $m = 5$ and $n = \{5,10,\ldots,40\}$. For each combination we solve 50 instances with a time limit of 600 seconds and show the average solution times. The results of this experiment is provided in Figures~\ref{ja-m5-time} and \ref{ja-m5-pp}.

The solution times presented in Figres~\ref{ja-m5-time} show that all introduced IPs perform similarly over the given generation methods. It can be seen that the problem constantly tends to be harder to solve from $n=5$ to $n=30$ and then solution times for all IPs decrease slightly. Here, IP-4 and IP-5 (which is one of the best also for the selection problem) are the fastest IPs. The reason is they have fewer number of constraints.

Like the solution times, Figres~\ref{ja-m5-pp} illustrates that IP-5 leads to the best performance profile, meaning that for most of instances it is the fastest IP following by IP-4. In this setting IP-2 is the worst IP in terms of solution times.

\begin{figure}[htbp]
	\begin{center}
		\subfigure[\texttt{Gen-1}]{\includegraphics[width=0.45\textwidth]{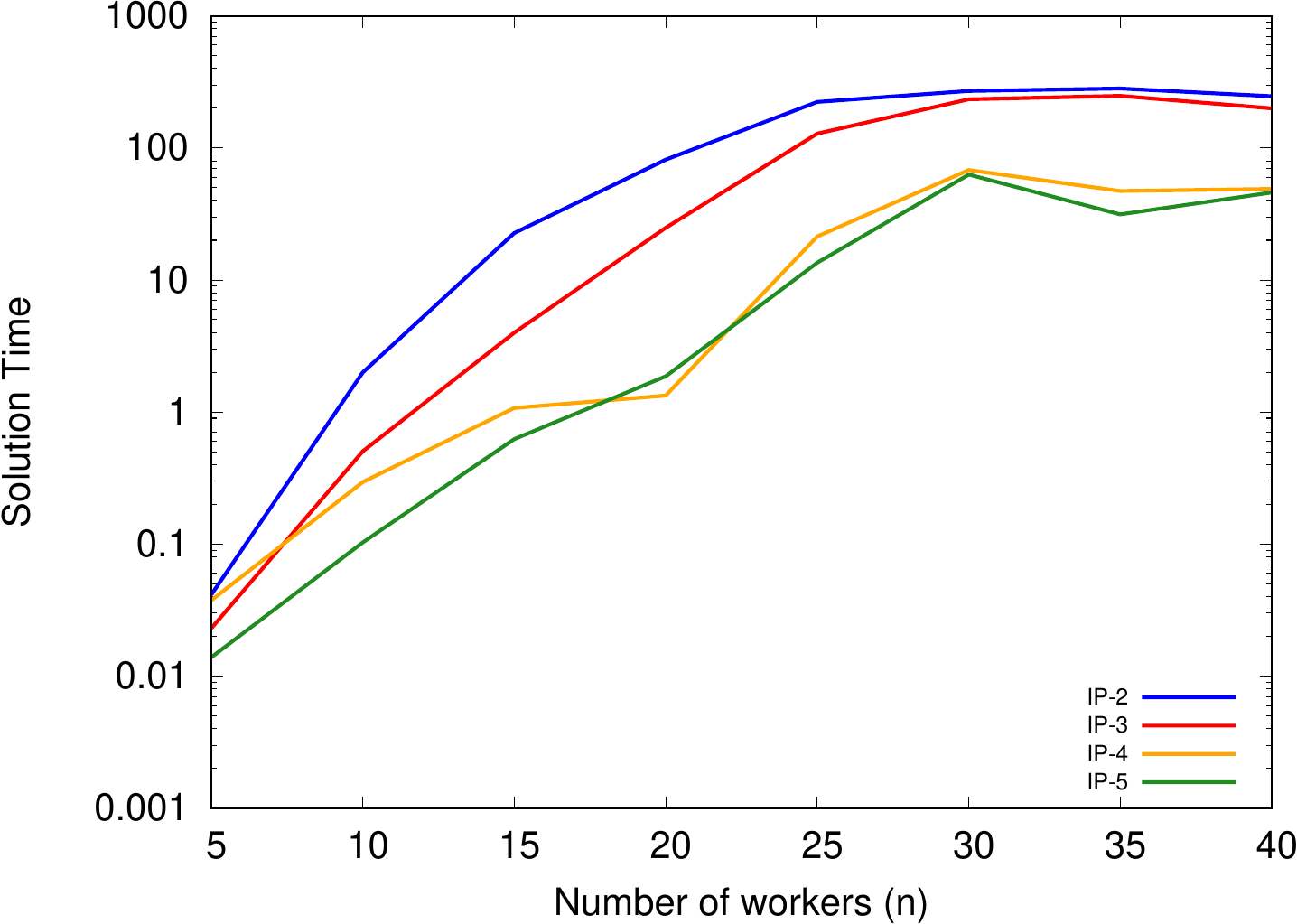}}
		\subfigure[\texttt{Gen-2}]{\includegraphics[width=0.45\textwidth]{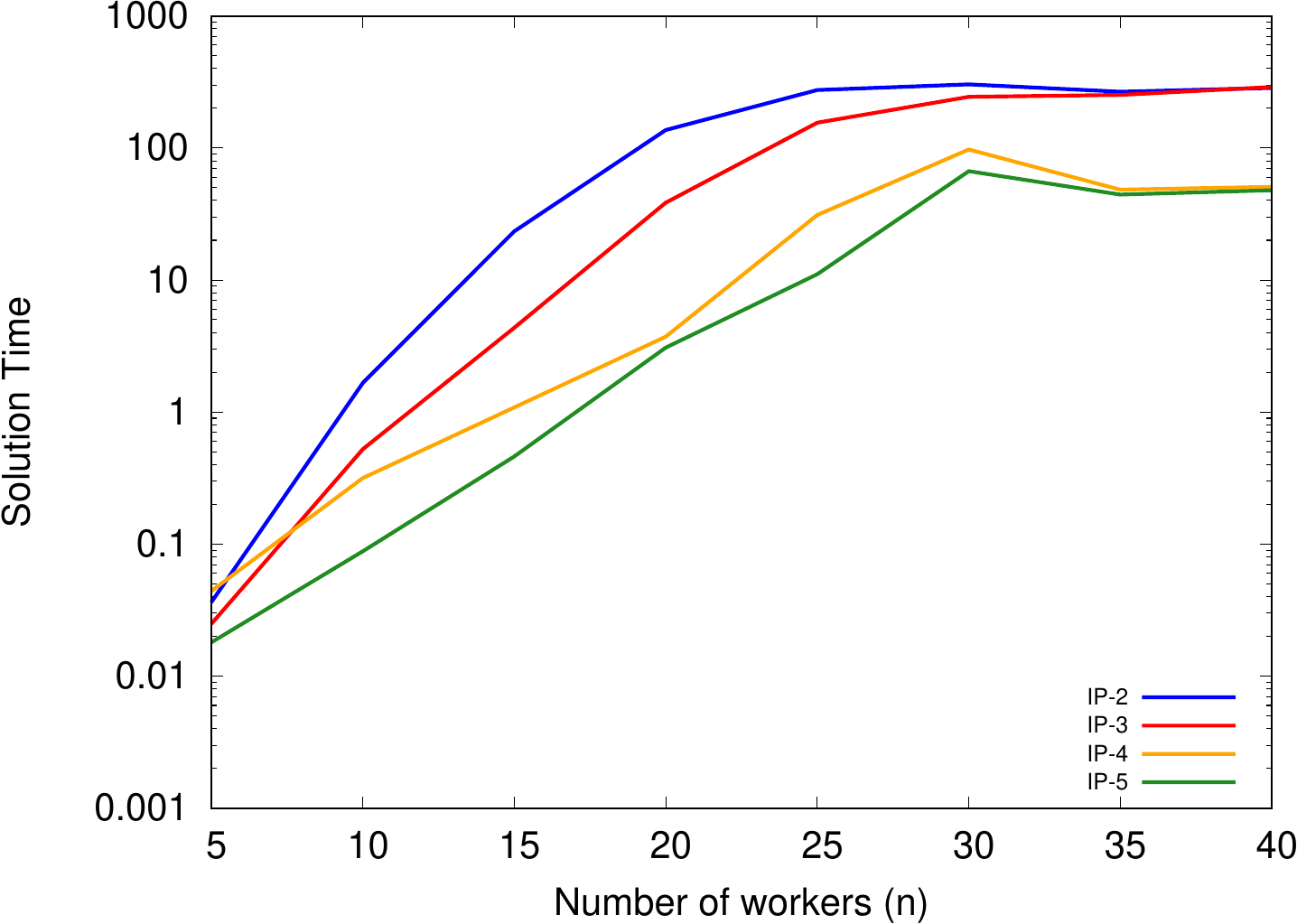}}
	\end{center}
	\caption{Job assignment - Exp1 - Solution times ($m=5$)}\label{ja-m5-time}
\end{figure}

\begin{figure}[htbp]
	\begin{center}
		\subfigure[\texttt{Gen-1}]{\includegraphics[width=0.45\textwidth]{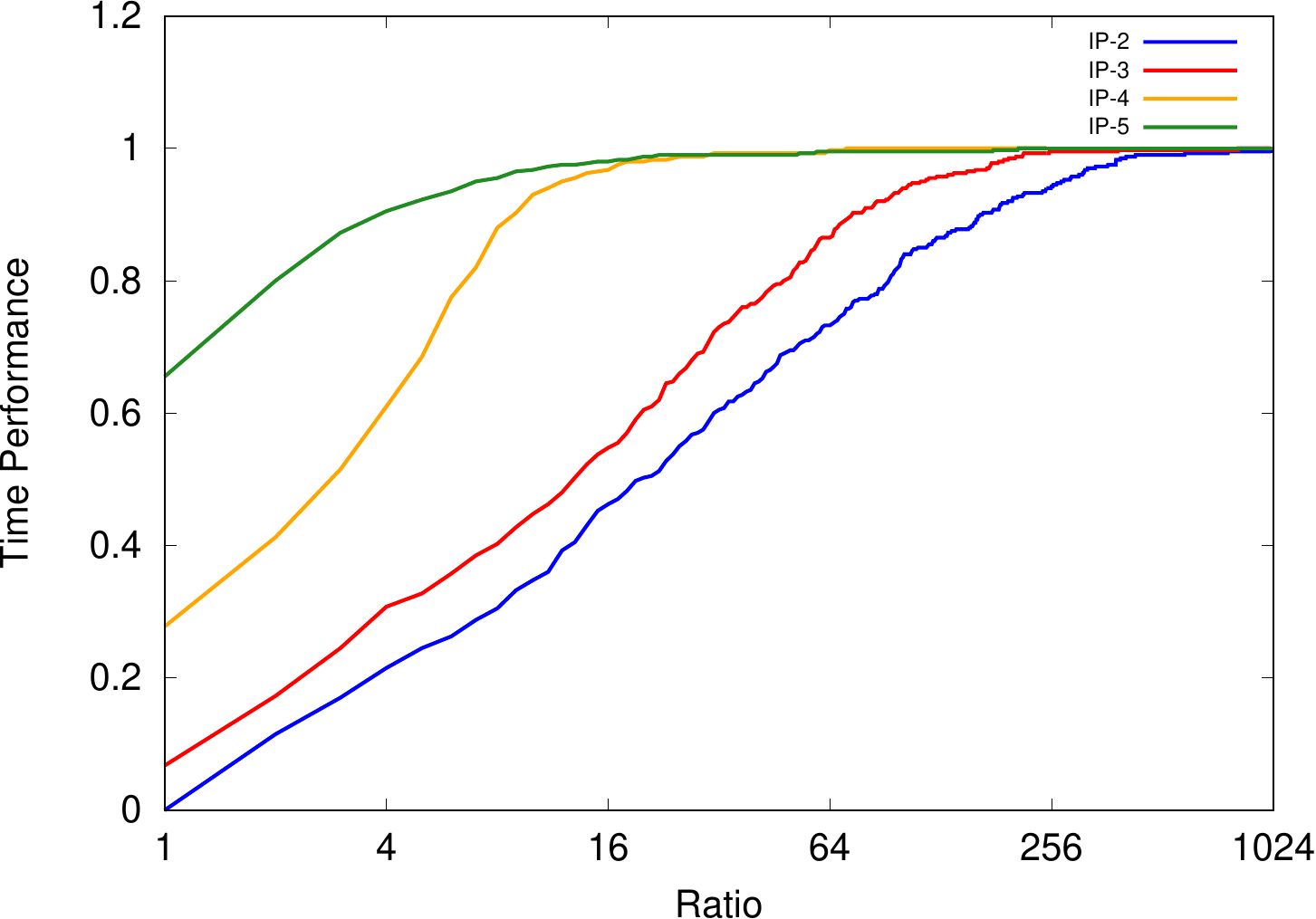}}
		\subfigure[\texttt{Gen-2}]{\includegraphics[width=0.45\textwidth]{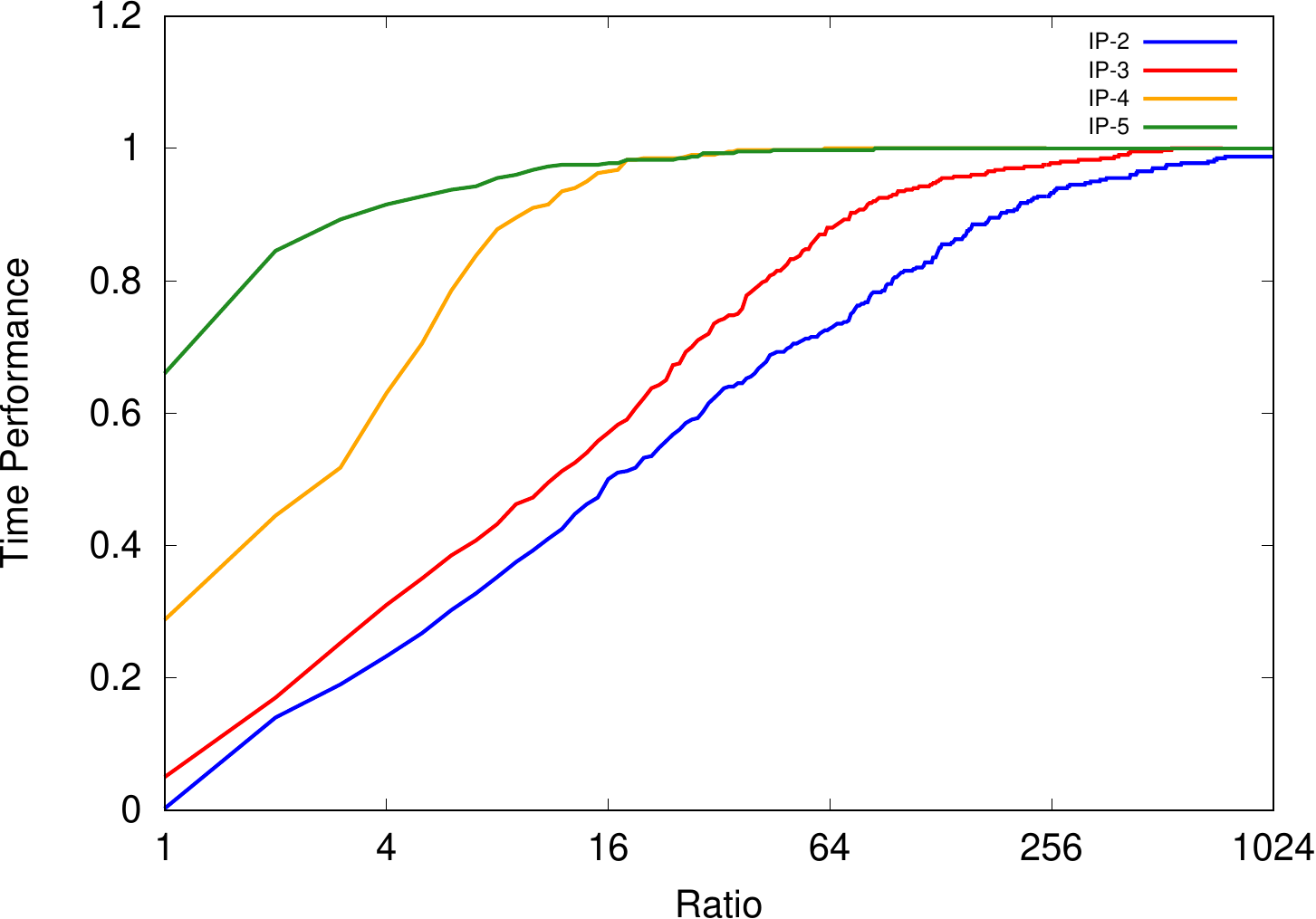}}
	\end{center}
	\caption{Job assignment - Exp1 - Performance profiles $m=5$}\label{ja-m5-pp}
\end{figure}

\subsubsection{Experiment 2}

In the second experiment of the job assignment problem, we fix the number of workers ($n=20$) and change the number of jobs ($n=\{2,3,\ldots,9\}$). Similarly, for each given combination we solve 50 instances with a time limit of 600 seconds and show the average solution times. The time performance of this experimental setting is shown in Figures~\ref{ja-n20-time} and \ref{ja-n20-pp}.

Here, again the same trend as experiment 1 (\ref{subsubsec:job-exp1}) can be observed in terms of both solution times and the corresponding performance profile. The results of the introduced generation methods are equivalent. However, the drop of the solution times for the larger case of instances are less noticeable. Likewise, IP-5 has both the best average and instance-wise solution times and the slowest IP is again IP-2.

\begin{figure}[htbp]
	\begin{center}
		\subfigure[\texttt{Gen-1}]{\includegraphics[width=0.45\textwidth]{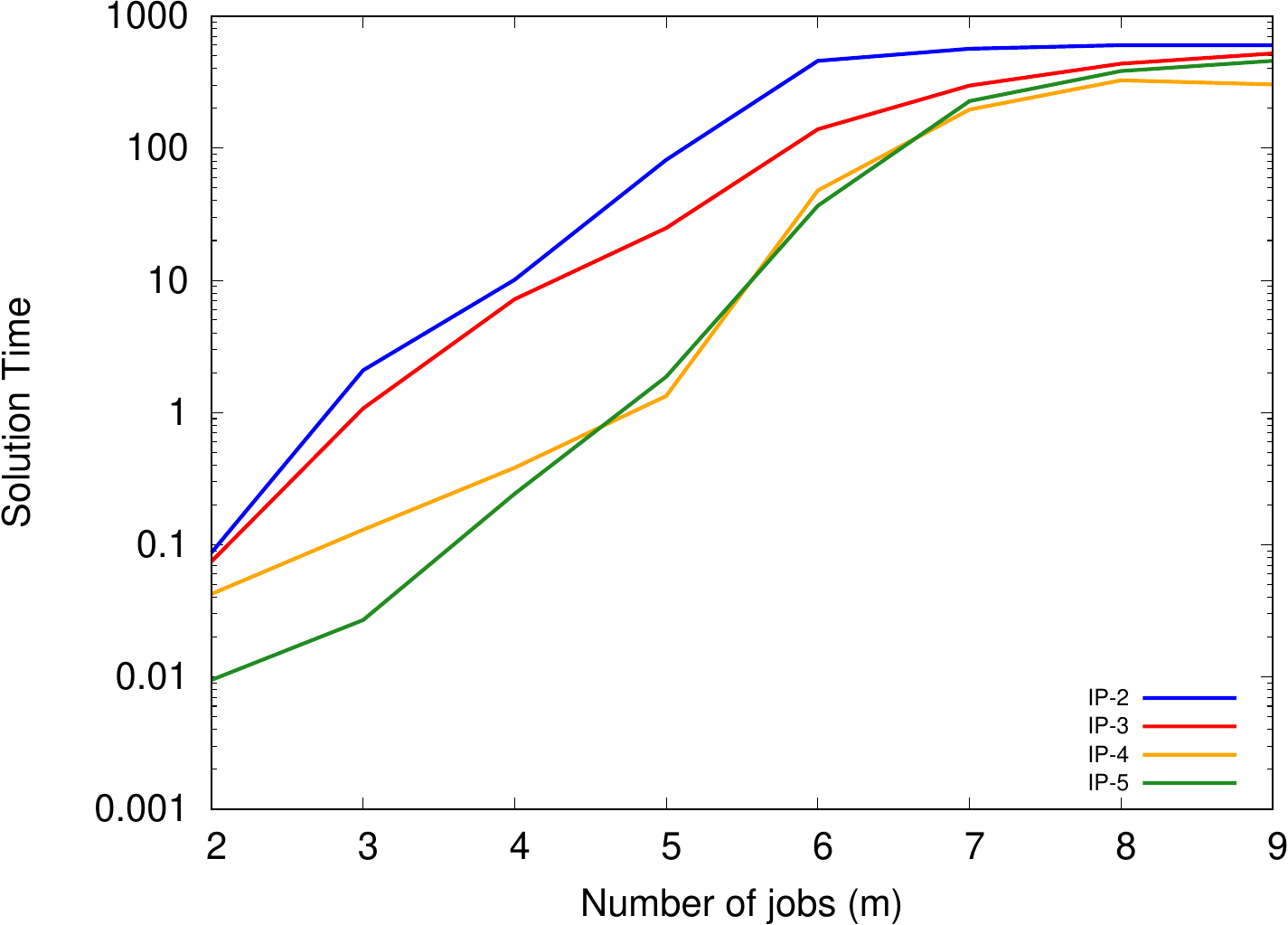}}
		\subfigure[\texttt{Gen-2}]{\includegraphics[width=0.45\textwidth]{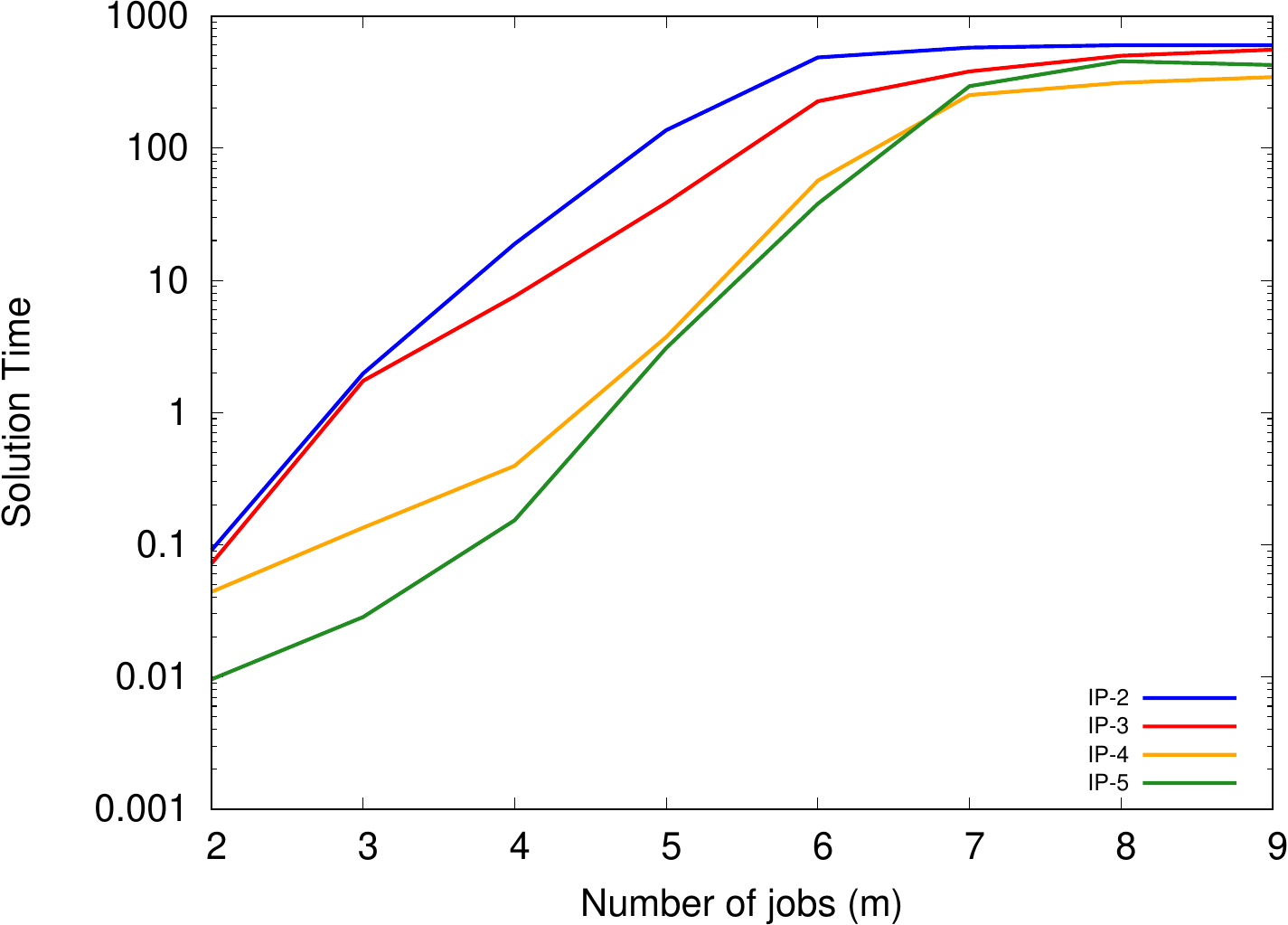}}
	\end{center}
	\caption{Job assignment - Exp2 - Solution times $n=20$}\label{ja-n20-time}
\end{figure}

\begin{figure}[htbp]
	\begin{center}
		\subfigure[\texttt{Gen-1}]{\includegraphics[width=0.45\textwidth]{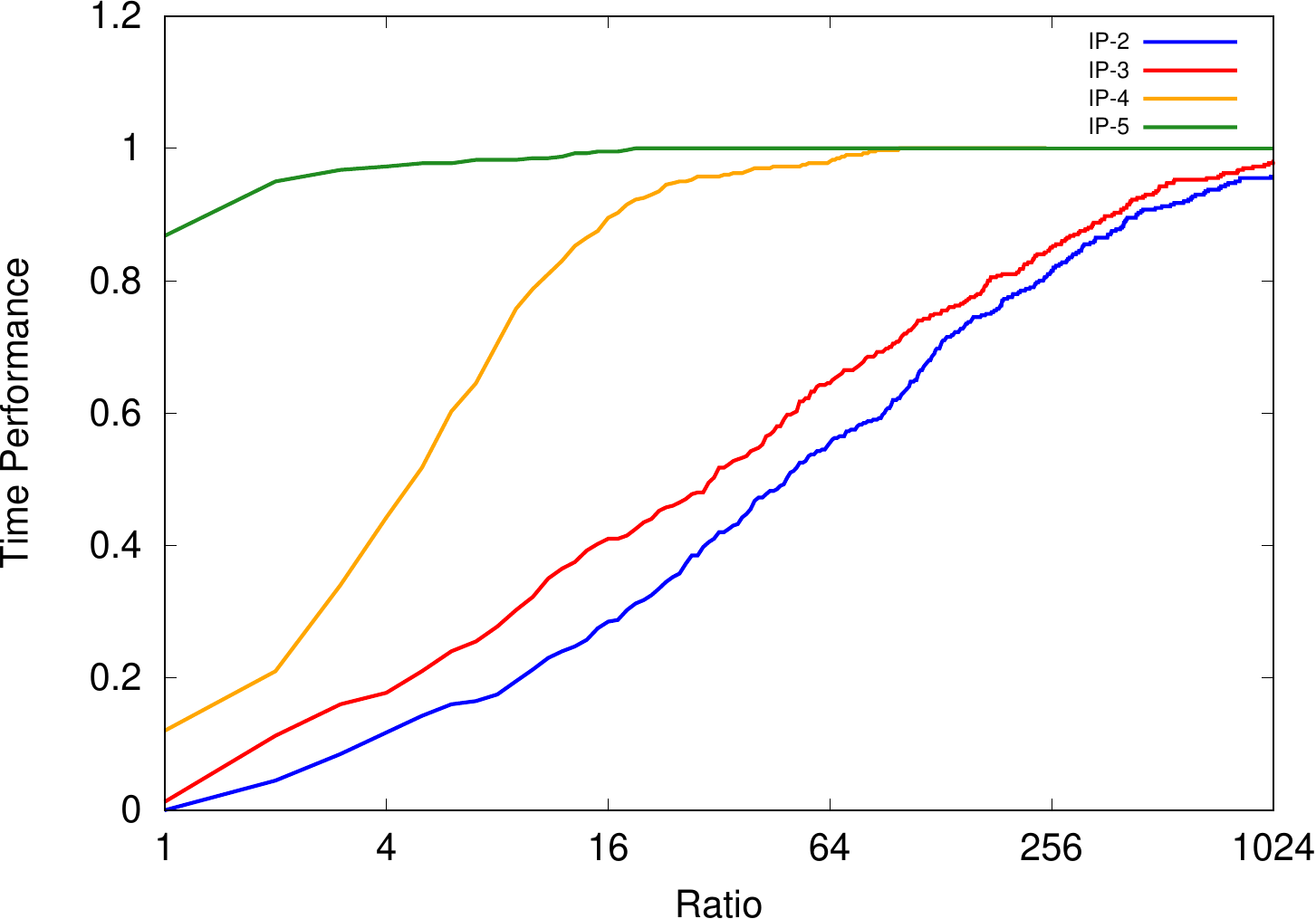}}
		\subfigure[\texttt{Gen-2}]{\includegraphics[width=0.45\textwidth]{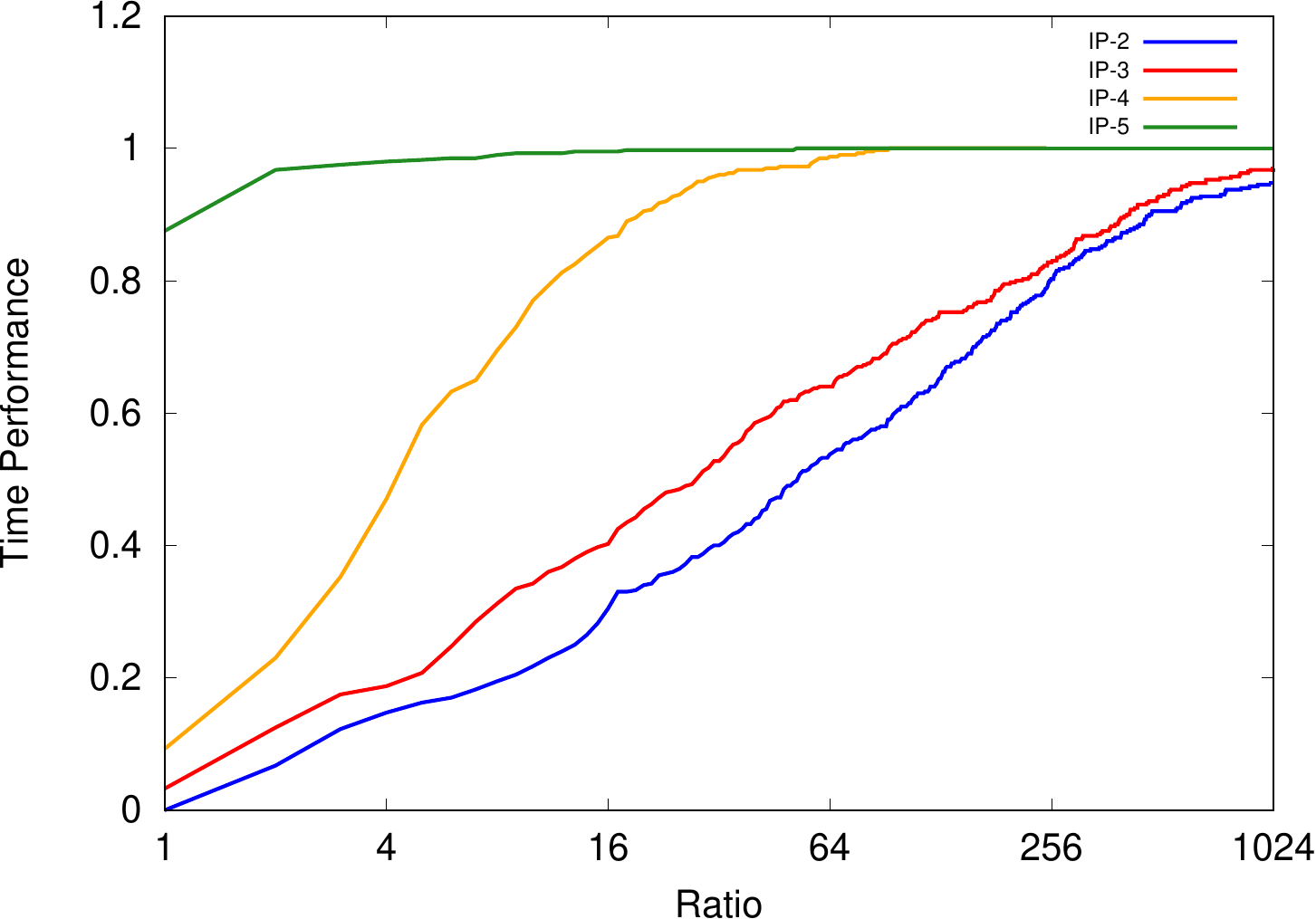}}
	\end{center}
	\caption{Job assignment - Exp2 - Performance profiles $n=20$}\label{ja-n20-pp}
\end{figure}

\subsection{2-Edge-Connected Spanning Subgraph}

\subsubsection{Setup}

In the NP-hard 2-edge-connected spanning subgraph problem, an undirected graph $G=(V,E)$ is given with edge weights $\pmb{d}\in\mathbb{R}^{|E|}_+$. The objective is to find a subset of edges $E'$ with minimum weight such that $G[E']$ is 2-edge-connceted, i.e., there are two edge-disjoint paths between any pair of nodes (see, e.g., \cite{huh2004finding}). We consider a robust version where edges can fail, but the subgraph is still required to remain 2-edge-connected.

The nominal problem can be formulated as follows:
\begin{align*}
\min\ &\sum_{\{i,j\}\in E} d_{ij} x_{ij} \\
\text{s.t.} & \sum_{\{i,j\}\in C} x_{ij} \ge 2 & \forall C\in\mathcal{C} \\
& x_{ij} \in \{0,1\} & \forall \{i,j\}\in E
\end{align*}
where $\mathcal{C}$ denotes the set of all cuts in the graph $G$. As there are exponentially many constraints, we use an iterative procedure. In Appendix~\ref{app:cutbased}, we describe models IP-2, IP-3, IP-4 and IP-5 for a subset $\mathcal{C}'\subseteq \mathcal{C}$ of cuts. We begine with $\mathcal{C}'=\emptyset$, solve the corresponding formulation, and check if the resulting solution $\pmb{x}$ is feasible with respect to all cuts $\mathcal{C}$. To check if a cut is violated, we solve the following IP:
\begin{align*}
\min & \sum_{e\in E} y_e + \epsilon z_e\\
\text{s.t.}\ & z_{ij} \ge u_i - u_j & \forall \{i,j\} \in E \\
& z_{ij} \ge u_j - u_i & \forall \{i,j\} \in E \\
& 1 \le \sum_{i\in[n]} u_i \le n-1 \\
& y_e \ge x_e + z_e - c_e - 1 & \forall e \in E \\
& \sum_{e\in E} w_e c_e \le B \\
& y_i, c_i, z_i \in \{0,1\} & \forall i \in [m] \\
& u_i \in \{0,1\} & \forall i \in [n]
\end{align*}
If a violated cut can be found, it is added to $\mathcal{C}'$, and the robust problem is solved again, until convergence is reached.

We introduce one experiment where the number of nodes of the given graph vary. We also set the density ($D$) of the graph equal to 0.8.

Unlike the experiments on both the selection and job assignment problem, we just introduce on approach of instance generation for the cut based problem under budgeted interdiction uncertainty. In this case our graph has $n$ nodes and $m$ edges, where $m = D \times \tfrac{n(n-1)}{2}$. Here, for each $e\in E$ we choose $d_e$ from $\{1,\ldots,10\}$ and $w_e$ from $\{5,\ldots,10\}$ randomly uniform. Moreover, we set $B=15$.

\subsubsection{Experiment}

In the only experiment of the cut based problem, we change the number of nodes and thus choose $n$ from $\{10,12,\ldots,30\}$. Similarly, for each instance size we solve 50 instances with a time limit of 600 seconds and show the average solution times. The time performance of this experimental setting is shown in Figure~\ref{cut-times}.

The results gathered in Figure~\ref{cut-times} show that the given IPs perform similarly for both the job assignment and the cut based problem. Here, again IP-5 is the fastest IP, whose performance profile also dominates the other methods. The second best is IP-4 and the worst performance belongs to IP-2.

\begin{figure}[htbp]
	\begin{center}
		\subfigure[Solution time]{\includegraphics[width=0.45\textwidth]{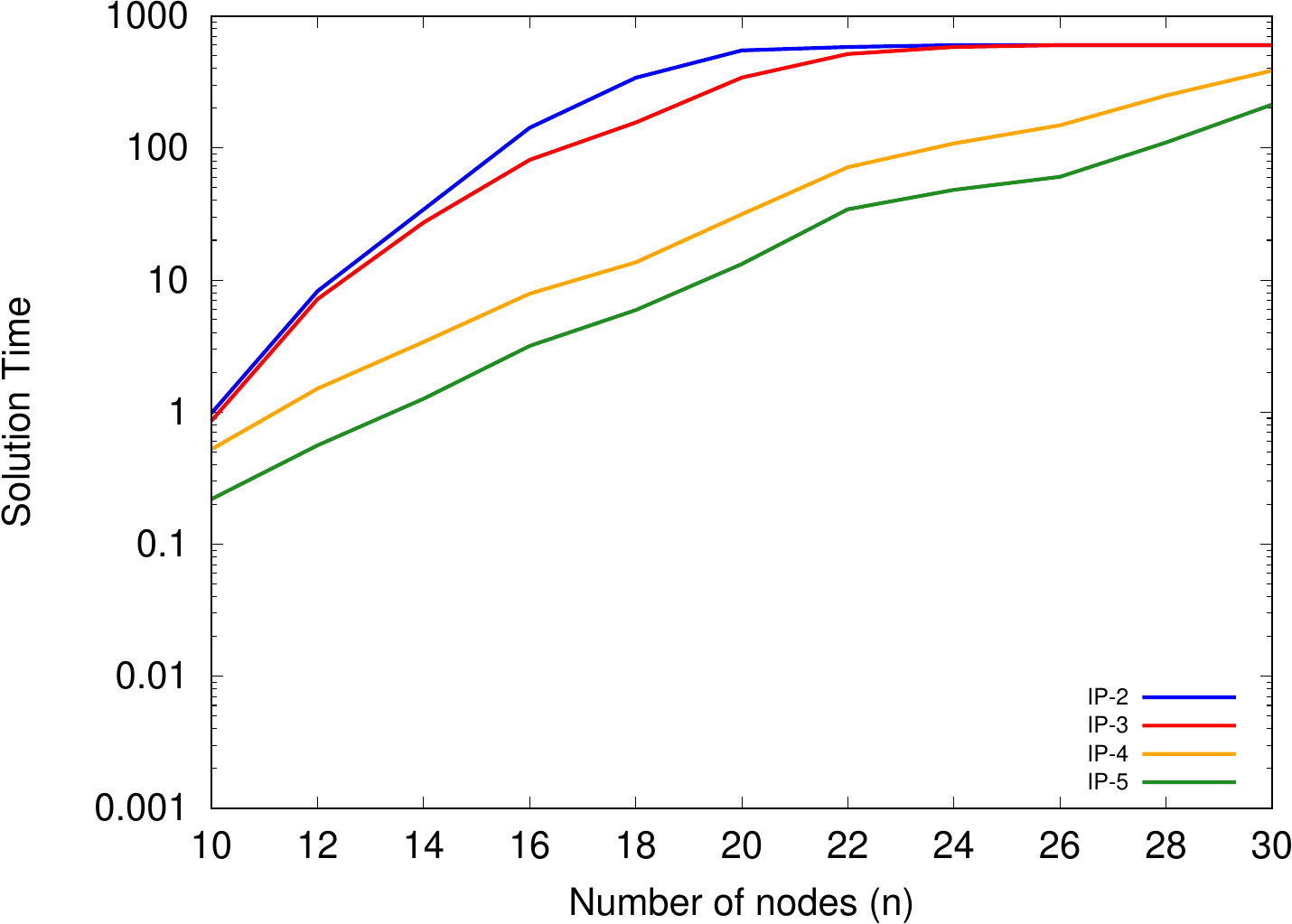}}
		\subfigure[Performance profile]{\includegraphics[width=0.45\textwidth]{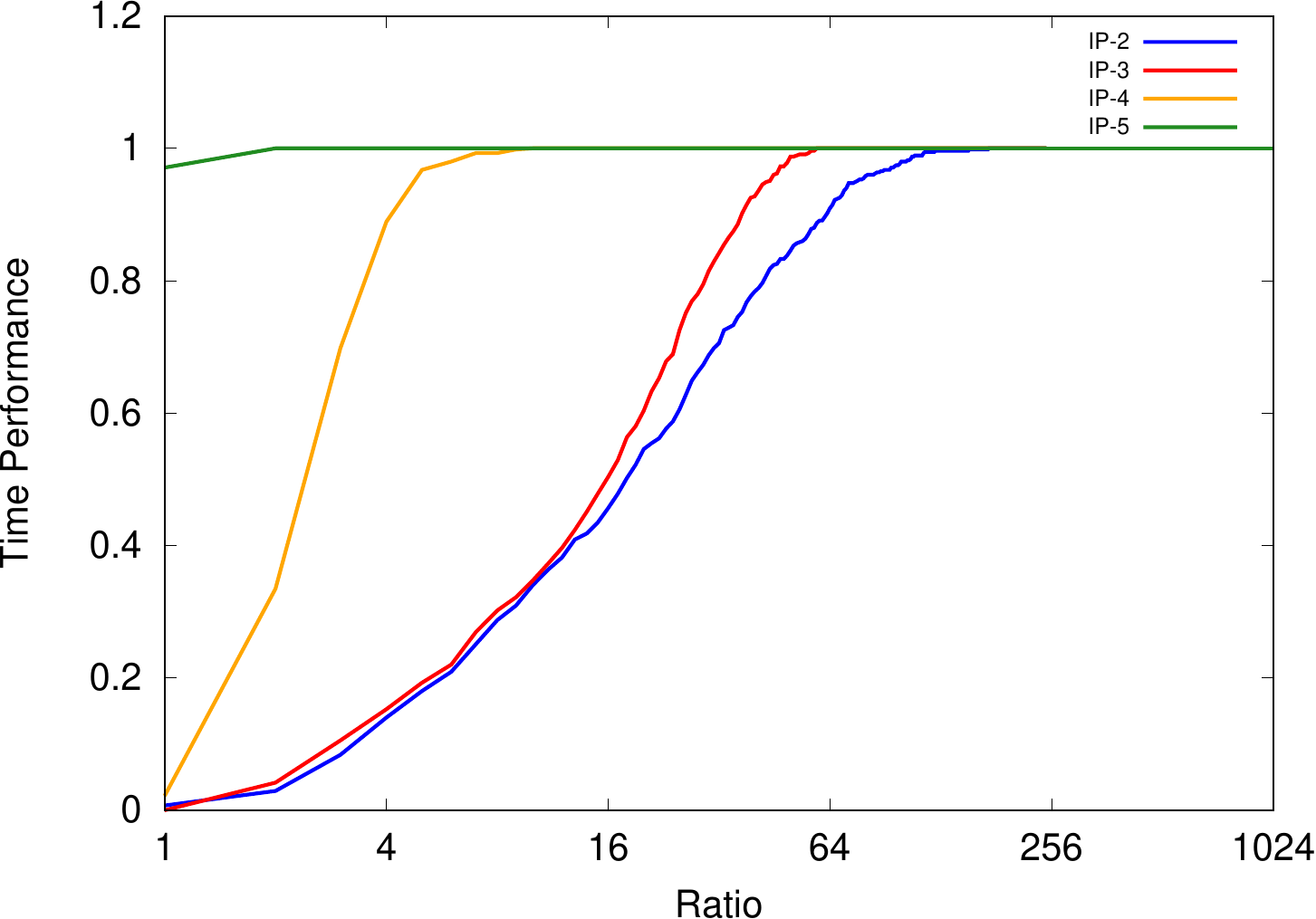}}
	\end{center}
	\caption{Spanning subgraph - Time performance ($D=0.8$)}\label{cut-times}
\end{figure}

\section{Conclusions}
\label{sec:conclusions}

Tools to model uncertainty sets are of central importance in robust optimization. While classic budgeted uncertainty sets make the assumption that each ``attack'' (i.e., changing a parameter away from its nominal value) has the same costs, extensions such as knapsack uncertainty sets relax this constraint and allow different attacks to have different costs. However, for a discrete knapsack uncertainty set, this means that even evaluating a solution (i.e., calculating the worst possible attack) requires us to solve an NP-hard knapsack problem.

In this paper, we propose an alternative uncertainty set, where attacks may have different costs, but each attack leads to the same consequence. Such a model is particularly useful to model uncertainty in cardinality-based constraints or objectives. This type of uncertainty has the advantage that calculating a worst-case attack is still possible in polynomial time, even though the corresponding robust problems become hard. We also demonstrated how this approach can be applied to a job assignment problem, and the problem of finding a minimum-cost 2-edge-connected subgraph under failures.

We consider different ways to model the worst-case attack, which lead to a total of five compact integer programming formulations, none of which dominates the other. Our computational experiments for the selection prblem indicate that in particular models IP-3 (based on rounding down attack budgets with integer variables) and IP-5 (based on forcing a binary variable to become active if an item can be attacked) show promising performance and can solve problems to proven optimality with up to 100 items within seconds. However, given the job assignment and finding minimum cost 2-edge-connected subgraph problem, IP-5 has the best performance followed by IP-4.

In further research, it will be interesting to explore if additive approximation algorithms may be possible, as multiplicative approximation guarantees are impossible to achieve. Furthermore, we intend to study budgeted interdiction sets in multi-stage environments such as two-stage and recoverable robust optimization.

\appendix

\section{Job Assignment Models}\label{app:jobassignment}

In this section we summarize the compact formulation of IP-2, IP-3, IP-4 and IP-5 that we obtained for the job assignment problem under the budgeted interdiction uncertainty set where the intermediate steps are excluded and only the final IPS are shown.

\subsection*{IP-2}
\begin{align*}
\max & \sum_{j\in[m]} p_j \bar{z}_j  \\ 
\text{s.t.} & \sum_{i\in[n]} x_{ij} - \beta_j \geq d_j \bar{z}_j & \forall j \in [m]\\
& \sum_{i\in[k]} (w_i \mu_{ijk} - x_{ij}) \geq B \alpha_{kj} - \beta_j & \forall j \in [m], k \in [n] \\
& \mu_{ijk} \leq k x_{ij} & \forall j \in [m], k\in[n], i\in[k]\\
& \mu_{ijk} \leq \alpha_{kj} & \forall j \in [m], k\in[n], i\in[k]\\
& \mu_{ijk} \geq 0 & \forall j \in [m], k\in[n], i\in[k]\\
& \alpha_{kj} \geq 0 & \forall j \in [m], k \in [n]\\
& \beta_j \geq 0 & \forall j \in [m]\\
& \pmb{x}\in\X
\end{align*}

\subsection*{IP-3}
\begin{align*}
\max & \sum_{j\in[m]} p_j \bar{z}_j \\
\text{s.t.} & \sum_{i\in[n]} x_{ij} - \alpha_j \geq d_j \bar{z}_j & \forall j \in [m]\\
& \sum_{i\in[k]} (x_{ij} - z_{ijk}) \leq \alpha_j & \forall j \in [m], k \in [n] \\
& z_{ijk} \leq y_{jk} & \forall j \in [m], k\in[n], i\in[k] \\
& z_{ijk} \leq \Bigg\lfloor \frac{\sum_{l\in[n]} w_l}{B} \Bigg\rfloor x_{ij} & \forall j \in [m], k\in[n], i\in[k] \\
& y_{jk} \leq \frac{\sum_{i\in[k]} w_i x_{ij}}{B} & \forall j \in [m], k\in[n] \\
& z_{ijk} \geq 0 & \forall j \in [m], k\in[n], i\in[k] \\
& \alpha_j \geq 0 & \forall j\in[m] \\
& \pmb{y} \in \mathbb{Z}^n \\
&\pmb{x}\in\X
\end{align*}

\subsection*{IP-4}
\begin{align*}
\max & \sum_{j\in[m]} p_j \bar{z}_j \\
\text{s.t.} & \sum_{i\in[n]} x_{ij} \geq d_j \bar{z}_j + \sum_{k\in[n]} y_{kj} & \forall j\in[m] \\
& (B + 1) y_{kj} + k\alpha^k_j - \sum_{i\in[n]} \beta^k_{ij} \geq B +1 & \forall j\in[m], k \in [n] \\
& \alpha^k_j - \beta^k_{ij} \leq w_i + (B+1)(1-x_{ij}) & \forall j\in[m], k\in[n], i\in[n] \\
& \alpha^k_j \geq 0 & \forall j\in[m], k\in[n] \\
& \beta^k_{ij} \geq 0 & \forall j\in[m], k\in[n], i\in[n] \\
& \pmb{y}\in\{0,1\}^n \\
& \pmb{x}\in\X
\end{align*}

\subsection*{IP-5}
\begin{align*}
\max & \sum_{j\in[m]} p_j \bar{z}_j\\
\text{s.t.} & \sum_{i\in[n]} x_{ij} \geq d_j \bar{z}_j + \sum_{k\in[n]} y_{kj} & \forall j\in[m] \\
& (B+1)(1-x_{jk}+y_{kj}) \ge B + 1 - \sum_{i\in[k]} w_i x_{ij} & \forall j\in[m], k\in[n] \\
& \pmb{y}\in\{0,1\}^n \\
& \pmb{x}\in\X
\end{align*}

\section{Cut-Based Models}\label{app:cutbased}

Let $\mathcal{C}'$ be a subset of all cuts. For each cut $C$, denote by $C=\{ e(C,1),\ldots, e(C,|C|)\}$ the corresponding edges, sorted by their weight $w_e$. We have $[|C|] = \{1,\ldots,|C|\}$.


\subsection*{IP-2}
\begin{align*}
\min & \sum_{e\in E} d_e x_e  \\
\text{s.t.} & \sum_{e \in C}  x_e \ge 2 + \beta_C & \forall C\in\mathcal{C}'\\
& \sum_{\ell\in[k]} w_{e(C,\ell)} \mu_{C\ell k} - x_{e(C,\ell)} \ge B \alpha_{Ck} - \beta_C & \forall C \in \mathcal{C}', k \in [|C|]\\
& \mu_{C\ell k} \le k x_{e(C,\ell)} & \forall C\in\mathcal{C}', k\in[|C|], \ell\in[k] \\
& \mu_{C\ell k} \le \alpha_{Ck} & \forall C\in\mathcal{C}', k\in[|C|] \\
& x_e \in \{0,1\} & \forall e \in E \\
& \beta_C \ge 0 & \forall C \in \mathcal{C}' \\
& \alpha_{Ck} \ge 0 & \forall C\in\mathcal{C}', k\in[|C|] \\
& \mu_{C\ell k} \ge 0 & \forall C\in\mathcal{C}', k\in[|C|], \ell\in[k]
\end{align*}

\subsection*{IP-3}
\begin{align*}
\min & \sum_{e\in E} d_e x_e  \\
\text{s.t.} & \sum_{e \in C}  x_e \ge 2 + \alpha_C & \forall C \in \mathcal{C}' \\
& \sum_{\ell \in [k]} x_{e(C,\ell)} - z_{C\ell k} \le \alpha_C & \forall C\in\mathcal{C}', k\in[|C|] \\
& z_{C\ell k} \le y_{Ck} & C\in\mathcal{C}', k\in[|C|], \ell\in[k] \\
& z_{C\ell k} \le M x_{e(C,\ell)} & \forall C\in\mathcal{C}', k\in[|C|], \ell\in[k] \\
& (B+\epsilon) y_{C,k} \le \sum_{\ell\in[k]} w_{e(C,\ell)} x_{e(C,\ell)} & \forall C\in\mathcal{C}', k\in[|C|] \\
& x_e \in \{0,1\} & \forall e\in E \\
& \alpha_C \ge 0 & \forall C\in\mathcal{C}' \\
& z_{C\ell k} \ge 0 & \forall C\in\mathcal{C}', k\in[|C|], \ell\in[k] \\
& y_{Ck} \in \mathbb{Z} & \forall C\in\mathcal{C}', k\in[|C|]
\end{align*}
where
\[M = \Bigg\lfloor \frac{\sum_{e\in E} w_e}{B+\epsilon} \Bigg\rfloor\]

\subsection*{IP-4}
\begin{align*}
\min & \sum_{e\in E} d_e x_e  \\
\text{s.t.} & \sum_{e\in C}  x_e \ge 2 + \sum_{k\in [|C|]} y_{Ck} & \forall C\in\mathcal{C}'\\
& (B+1) y_{Ck} + k \alpha_{Ck} - \sum_{\ell\in [|C|]} \beta_{Ck\ell} \ge B + 1  & \forall C\in\mathcal{C}', k\in[|C|] \\
& \alpha_{Ck} - \beta_{Ck\ell} \le w_{e(C,\ell)} + (B+1) (1-x_{e(C,\ell)})  & \forall C\in\mathcal{C}', k\in[|C|], \ell\in[|C|] \\
& x_e \in \{0,1\} & \forall e \in E \\
& y_{Ck} \in \{0,1\} & \forall C\in\mathcal{C}', k\in[|C|] \\
& \alpha_{Ck} \ge 0 & \forall  C\in\mathcal{C}', k\in[|C|] \\
& \beta_{Ck\ell} \ge 0 & \forall C\in\mathcal{C}', k\in[|C|], \ell\in[|C|]
\end{align*}

\subsection*{IP-5}
\begin{align*}
\min & \sum_{e\in E} d_e x_e  \\
\text{s.t.} & \sum_{e\in C}  x_e \ge 2 + \sum_{e \in C} y_{Ce} & \forall C\in\mathcal{C}' \\
& (B+1) (1-x_{e(C,k)} +  y_{Ce(C,k)}) \ge (B + 1) - \sum_{\ell\in[k]} w_{e(C,\ell)} x_{e(C,\ell)} & \forall C\in\mathcal{C}', k\in[|C|] \\
& x_e \in \{0,1\} & \forall e \in E \\
& y_{Ce} \in \{0,1\} & \forall C\in\mathcal{C}', e \in C\\
\end{align*}

\end{document}